\pdfoutput=1

\documentclass[final]{siamltex}

\usepackage{amsfonts}
\usepackage[leqno]{amsmath}
\usepackage{ amssymb }
\usepackage{tikz}
\usepackage{algorithmic}
\usepackage{url}
\usepackage{multirow}
\usepackage{colortbl}
\usepackage{booktabs}
\usepackage{siunitx}

\newcommand{\ten}[1]{\mathcal{#1}}  
\DeclareMathOperator{\spn}{opspan}

\newtheorem{algorithm}{Algorithm}[section]

\title{A Constructive Algorithm for Decomposing a Tensor into a Finite Sum of Orthonormal Rank-1 Terms \thanks{This work was supported in part by the Hong Kong Research Grants Council under General Research Fund (GRF) Projects 718213E and 17208514, and the University Research Committee of The University of Hong Kong.}}

\author{Kim Batselier, Haotian Liu,
        \and Ngai Wong \thanks{Department of Electrical and Electronic Engineering, The University of Hong Kong, Hong Kong}}

\begin{document}

\maketitle

\begin{abstract}
We propose a constructive algorithm that decomposes an arbitrary real tensor into a finite sum of orthonormal rank-1 outer products. The algorithm, named TTr1SVD, works by converting the tensor into a tensor-train rank-1 (TTr1) series via the singular value decomposition (SVD). TTr1SVD naturally generalizes the SVD to the tensor regime with properties such as uniqueness for a fixed order of indices, orthogonal rank-1 outer product terms, and easy truncation error quantification. Using an outer product column table it also allows, for the first time, a complete characterization of all tensors orthogonal with the original tensor. Incidentally, this leads to a strikingly simple constructive proof showing that the maximum rank of a real $2 \times 2 \times 2$ tensor over the real field is 3. We also derive a conversion of the TTr1 decomposition into a Tucker decomposition with a sparse core tensor. Numerical examples illustrate each of the favorable properties of the TTr1 decomposition.%
\end{abstract}



\pagestyle{myheadings}
\thispagestyle{plain}
\markboth{KIM BATSELIER, HAOTIAN LIU, NGAI WONG}{TTr1 DECOMPOSITION}

\section{Introduction}
There has been a recent surge in the research and utilization of tensors, which are high-order generalization of matrices, and their low-rank approximations~\cite{TTB_Dense, candecomp, tensorreview, ivanTT,SAND2006-2081}. This is due to their natural form to capture high dimensional problems and their efficient compact representation of large-scale data sets.

Among various tensor decompositions, the CANDECOMP/PARAFAC (CP) decomposition\footnote{Originally introduced by Hitchcock \cite{Hitchcock1927}, the decomposition was rediscovered independently as CANDECOMP (CANonical DECOMPosition) by Carroll and Chang~\cite{candecomp}, and PARAFAC (PARAllel FACtors) by Harshman~\cite{harshman1970fpp}. The underlying algorithms are however the same.}~\cite{candecomp, harshman1970fpp, tensorreview} has found widespread use. CP expresses a tensor as the sum of a finite number of rank-1 tensors, called outer products, so that the tensor (CP-)rank can be defined as the minimum number of terms in the decomposition. Although CP is regarded as the generalization of the matrix singular value decomposition (SVD) to tensors, unlike matrices, there are no feasible algorithms to determine the rank of a specific tensor. Furthermore, most existing CP algorithms are optimization-based, such as the ``workhorse'' algorithm for CP: the alternating least squares (ALS)-CP method~\cite{candecomp}. ALS-CP minimizes the error between the original tensor and its rank-$R$ approximation (viz., sum of $R$ outer products) in an iterative procedure. The main problem of ALS-CP is that it only works by prescribing the rank $R$, therefore the procedure itself does not directly identify the tensor rank. Moreover, the outer products generated by ALS-CP are not orthogonal with each other unlike the case for matrix singular vectors.

Other tensor decompositions, for example the Tucker decomposition~\cite{candecomp, tuckerreview}, compress a tensor into a core tensor and several factor matrices. The Tucker decomposition of a tensor is not unique. One of its realizations can be efficiently computed by the higher-order SVD (HOSVD)~\cite{Lathauwer:HOSVD}. Each element in its core tensor can be deemed as the weight of a rank-1 factor. In this interpretation, all rank-1 factors of the Tucker decomposition are orthonormal. Nonetheless, the Tucker decomposition is not necessarily canonical and therefore cannot be used to estimate tensor ranks. 

To this end, a constructive orthogonal tensor decomposition algorithm, named tensor-train rank-1 (TTr1) SVD or TTr1SVD, is proposed in this paper. The recent introduction of the tensor-train (TT) decomposition~\cite{ivanTT} provides a constructive approach to represent and possibly compress tensors. Similar to the TT decomposition, the TTr1 decomposition reshapes and factorizes the tensor in a recursive way. However, unlike the TT decomposition, one needs to progressively reshape and compute the SVD of each singular vector to produce the TTr1 decomposition. The resulting singular values are constructed into a tree structure whereby the product of each branch is the weight of one orthonormal (rank-1) outer product. Most of the main properties and contributions of the TTr1 decomposition are highly reminiscent of the matrix SVD:
\begin{enumerate}
\item an arbitrary tensor is for a fixed order of the indices \emph{uniquely} decomposed into a linear combination of orthonormal outer products, each associated with a non-negative TTr1 singular value,
\item the approximation error of an $R$-term approximation is easily quantified in terms of the singular values,
\item numerical stability of the algorithm due to the use of consecutive SVDs,
\item characterizes the orthogonal complement tensor space that contains all tensors whose inner product is 0 with the original tensor $\ten{A}$. This orthogonal complement tensor space is, to our knowledge, new in the literature,
\item straightforward conversion of the TTr1 decomposition into the Tucker format with a sparse core tensor and orthogonal matrix factors.
\end{enumerate}

Having developed TTr1SVD, we found that its core routine turns out to be an independent re-derivation of the PARATREE algorithm~\cite{salmi2009sequential}. However, TTr1SVD bears the physical insight of enforcing a rank-1 constraint onto the TT decomposition~\cite{ivanTT}. Such a TT rank-1 perspective provides a much more straightforward appreciation of the favorable properties of this orthogonal SVD-like tensor decomposition. In particular, we provide a significantly more in-depth treatment of TTr1 decomposition than in~\cite{salmi2009sequential}, leading to important \emph{new results} such as a perturbation analysis of the singular values, a direct conversion of the TTr1 to the Tucker format featuring a sparse core tensor, and a full characterization of orthogonal complement tensors. Specifically, we introduce a TTr1-based tabulation of all orthogonal outer products that span a tensor $\ten{A}$, as well its orthogonal complement space $\spn(\ten{A})^\perp$ that is proposed for the first time in the literature. This permits, as an immediate application, an elegant and constructive proof that the rank of a real $2 \times 2 \times 2$ tensor over the real field is maximally 3. A Matlab/Octave implementation of our TTr1SVD algorithm can be freely downloaded and modified from \url{https://github.com/kbatseli/TTr1SVD}.

The outline of this paper is as follows. First, we introduce some notations and definitions in Section \ref{sec:notations}. Section \ref{sec:TTr1} presents a brief overview of the TT decomposition together with a detailed explanation of our TTr1 decomposition. Properties of the TTr1 decomposition such as uniqueness, orthogonality, approximation errors, orthogonal complement tensor space, perturbation of singular values and Tucker conversion are discussed in Section~\ref{sec:properties}. These properties are illustrated in Section \ref{sec:examples} by means of several numerical examples. Section \ref{sec:conclusions} concludes and summarizes the contributions.

\subsection{Notation and definitions}
\label{sec:notations}
We will adopt the following notational conventions. A $d$th-order tensor, assumed real throughout this paper, is a multi-way array $\ten{A} \in \mathbb{R}^{n_1\times n_2 \times \cdots \times n_d}$ with elements $\ten{A}_{i_1i_2\cdots i_d}$ that can be perceived as an extension of a matrix to its general $d$th-order, also called $d$-way, counterpart. We consider only real tensors because we adopt an application point of view. This is however without loss of generality, one could easily consider tensors over $\mathbb{C}$, which would require the replacement of the transpose by the conjugate transpose. Although the wordings `order' and `dimension' seem to be interchangeable in the tensor community, we prefer to call the number of indices $i_k\,(k=1,\ldots,d)$ the order of the tensor, while the maximal value $n_k\,(k=1,\ldots,d)$ associated with each index the dimension. A cubical tensor is a tensor for which $n_1=n_2=\cdots=n_d=n$. The $k$-mode product of a tensor $\ten{A} \in \mathbb{R}^{n_1\times n_2 \times \cdots \times n_d}$ with a matrix $U \in \mathbb{R}^{p_k \times n_k}$ is defined by
$$
(\ten{A} {_{\times_k}} U)_{i_1\cdots i_{k-1} j_k i_{k+1} \cdots i_d} \;=\; \sum_{i_k=1}^{n_k} U_{j_k i_k} \ten{A}_{i_1 \cdots i_k\cdots i_d},
$$
so that $\ten{A}{_{\times_k}} U \in \mathbb{R}^{n_1 \times \cdots \times n_{k-1} \times p_k \times n_{k+1} \times \cdots \times n_d}$. The inner product between two tensors $\ten{A},\ten{B} \in \mathbb{R}^{n_1\times \cdots \times n_d}$ is defined as
$$
\langle \ten{A},\ten{B} \rangle \;=\; \sum_{i_1,i_2,\cdots,i_d}\,\ten{A}_{i_1i_2 \cdots i_d}\,\ten{B}_{i_1i_2 \cdots i_d}.
$$
The norm of a tensor is taken to be the Frobenius norm $||\ten{A}||_F=\langle \ten{A},\ten{A} \rangle^{1/2}$. The vectorization of a tensor $\ten{A}$, denoted $\textrm{vec}(\ten{A}) \in \mathbb{R}^{n_1\cdots n_d}$, is the vector obtained from taking all indices together in one mode. A third-order rank-1 tensor can always be written as the outer product~\cite{tensorreview}
$$
\sigma\, (a \circ b \circ c) \quad \textrm{with components } \quad \ten{A}_{i_1 i_2 i_3}\;=\;\sigma\,a_{i_1}\,b_{i_2}\,c_{i_3}
$$
with $\sigma \in \mathbb{R}$ whereas $a$, $b$ and $c$ are vectors of arbitrary lengths as depicted in Figure~\ref{fig:rank1tensor}. Similarly, any $d$-way rank-1 tensor can be written as an outer product of $d$ vectors. Using the $k$-mode multiplication, this outer product can also be written as $\sigma {_{\times_1}} a {_{\times_2}} b {_{\times_3}}  c$ where $\sigma$ is now regarded as a $1\times 1\times 1$ tensor. In order to facilitate the discussion of the TTr1 decomposition we will make use of a running example tensor $\ten{A} \in \mathbb{R}^{3\times 4 \times 2}$ shown in Figure~\ref{fig:A}.%
\begin{figure}[ht]
\centering
\includegraphics[width=.18\textwidth]{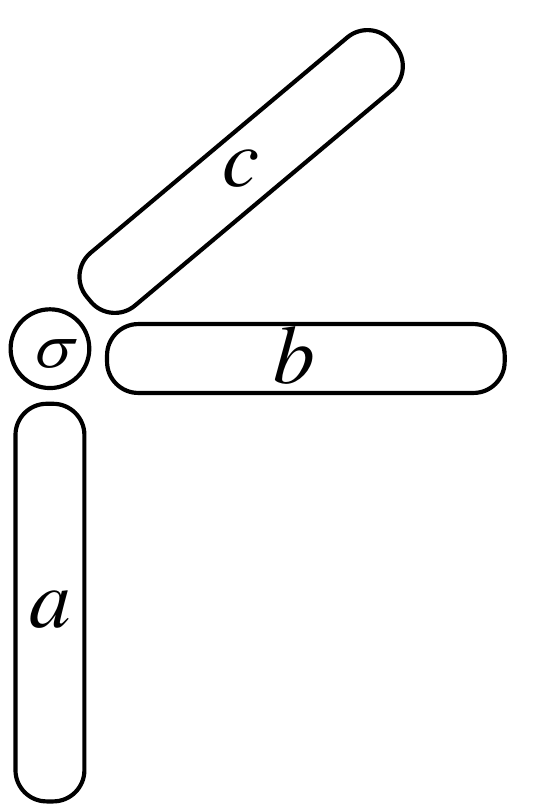}
\caption{The outer product of 3 vectors $a,b,c$ of arbitrary lengths forming a rank-1 outer product.}
\label{fig:rank1tensor}
\end{figure}
\begin{figure}[ht]
\begin{center}
\includegraphics[width=0.4\textwidth]{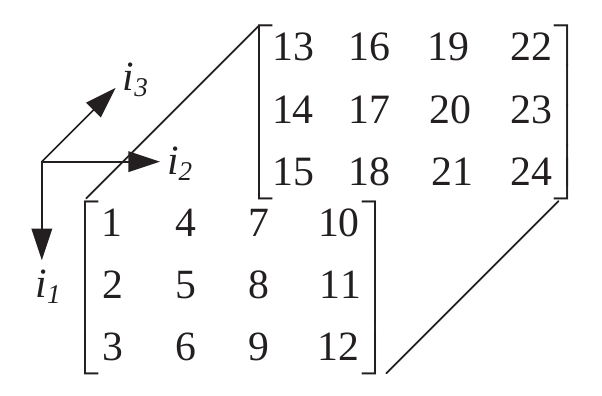}
\caption{A running $3\times 4\times 2$ tensor example.}
\label{fig:A}
\end{center}
\end{figure}

\section{TTr1 decomposition}
\label{sec:TTr1}
\subsection{TT decomposition}
Our decomposition is directly inspired by the TT decomposition \cite{ivanTT}, which we will succinctly review here. The main idea of the TT decomposition is to re-express a tensor $\ten{A}$ as
\begin{equation}
\ten{A}_{i_1 i_2 \cdots i_d} = \ten{G}_1(i_1) \, \ten{G}_2(i_2) \cdots \ten{G}_d(i_d),
\label{eq:TTdecomp}
\end{equation}
where for a fixed $i_k$ each $\ten{G}_k(i_k)$ is an $r_{k-1} \times r_k$ matrix, also called the TT core. Note that the subscript $k$ of a core $\ten{G}_k$ indicates the $k$th core of the TT decomposition. The ranks $r_k$ are called the TT ranks. Each core $\ten{G}_k$ is in fact a third-order tensor with indices $\alpha_{k-1},i_k,\alpha_k$ and dimensions $r_{k-1},n_k,r_k$, respectively. Since $\ten{A}_{i_1 i_2 \cdots i_d}$ is a scalar we obviously have that $r_0=r_d=1$ and for this reason $\alpha_0$ and $\alpha_d$ are omitted. Consequently, we can write the elements of $\ten{A}$ as
\begin{equation}
\ten{A}_{i_1 i_2 \cdots i_d} = \sum_{\alpha_1,\cdots,\alpha_{d-1}} \ten{G}_1(i_1 ,\alpha_1) \, \ten{G}_2(\alpha_1, i_2, \alpha_2) \cdots \ten{G}_d(\alpha_{d-1} ,i_d),
\label{eq:TTcomp}
\end{equation}
where we always need to sum over the auxiliary indices $\alpha_1,\ldots,\alpha_{d-1}$, and therefore \eqref{eq:TTcomp} is equivalent to the matrix product form in \eqref{eq:TTdecomp}. An approximation of $\ten{A}$ is achieved by truncating the $\alpha_k$ indices in \eqref{eq:TTcomp} at values smaller than the TT ranks $r_k$.

\begin{figure}[ht]
\centering
\includegraphics[width=.7\textwidth]{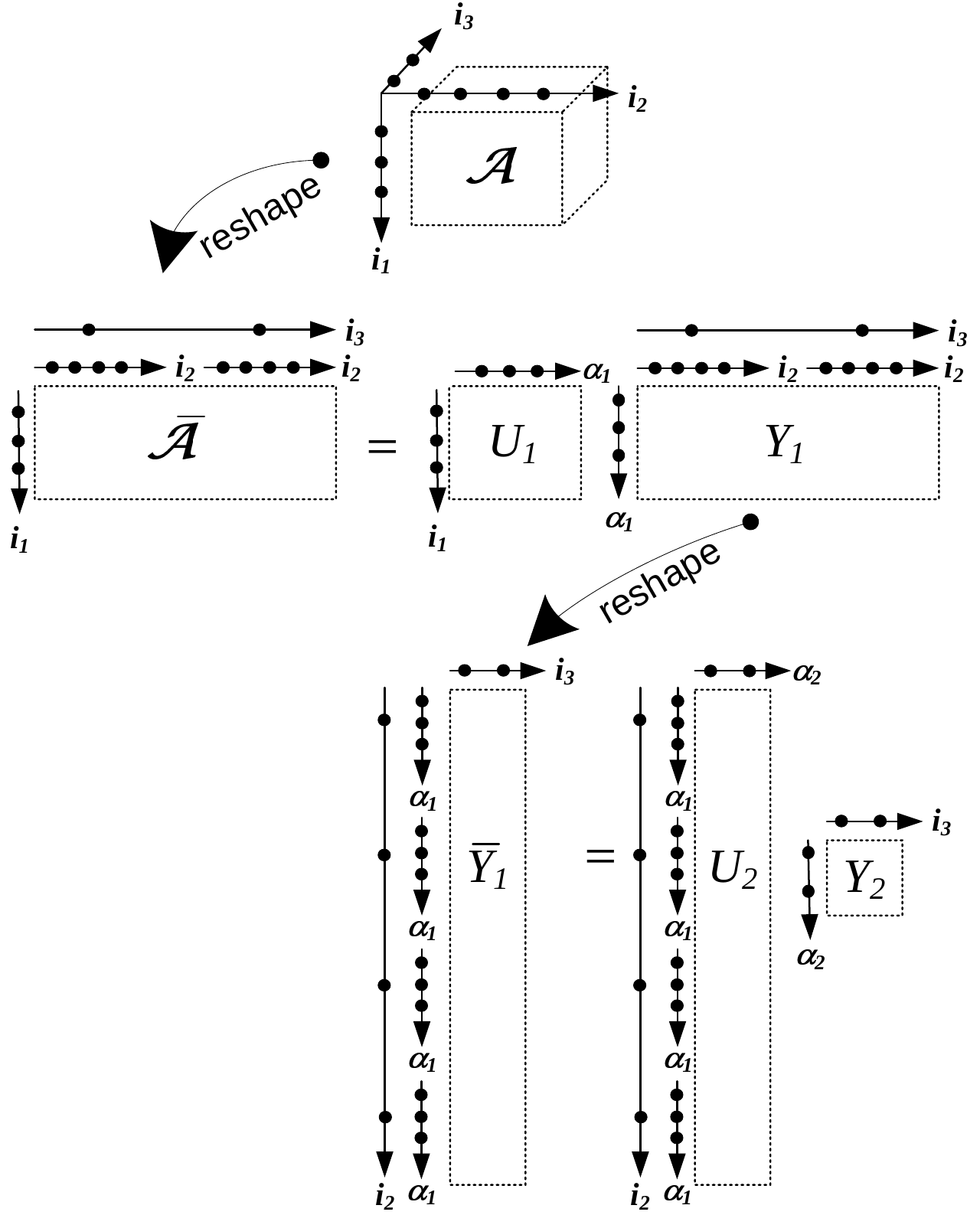}
\caption{{Computation of the TT decomposition of $\ten{A}$.}}
\label{fig:TTA1}
\end{figure}
\begin{figure}[ht]
\centering
\includegraphics[width=0.55\textwidth]{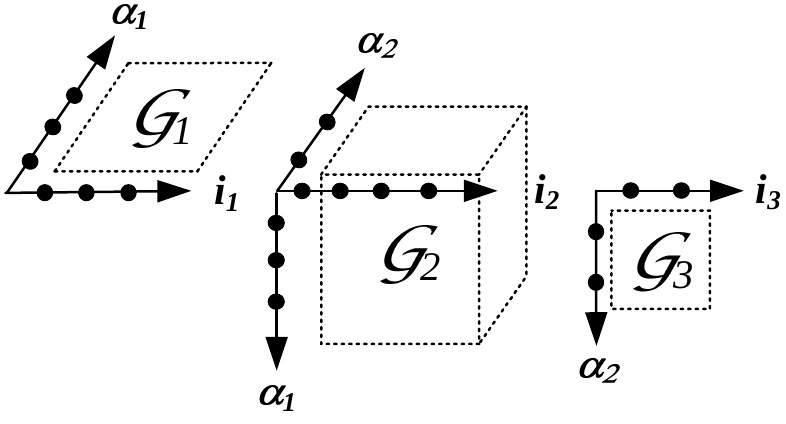}
\caption{{TT decomposition of $\ten{A}$. Each set of orthogonal axes represents a core $\ten{G}_k$ of the TT.}}
\label{fig:TTA}
\end{figure}

Computing the TT decomposition consists of doing $d-1$ consecutive reshapings and SVD computations. For our running example $\ten{A} \in \mathbb{R}^{3\times 4 \times 2}$ in Figure~\ref{fig:A}, this means that the decomposition is computed in 2 steps. This process is visualized in Figure~\ref{fig:TTA1}, whereby $\ten{A}$ will eventually be converted into its TT format in Figure~\ref{fig:TTA}. Referring to Figure~\ref{fig:TTA1}, the first reshaping of $\ten{A}$ into a matrix $\bar{\ten{A}}$ that needs to be considered is by grouping the indices $i_2,i_3$ together. This results in the $3\times 8$ matrix
$$
\bar{\ten{A}} \;=\;
\begin{pmatrix}
1 & 4 & 7 & 10 & 13 & 16 & 19 & 22\\
2 & 5 & 8 & 11 & 14 & 17 & 20 & 23\\
3 & 6 & 9 & 12 & 15 & 18 & 21 & 24
\end{pmatrix}.
$$ 
The ``economical'' SVD of the $3 \times 8$ matrix $\bar{\ten{A}}$ is then
\begin{equation}
\bar{\ten{A}}  \;=\; U_1\; S_1 \; V_1^T,
\label{ex:SVD1}
\end{equation}
with $U_1$ a $3\times 3$ matrix and $V_1$ an $8\times 3$ matrix. In fact, any dyadic decomposition can be used for this step in the TT algorithm, but the SVD is often chosen for its numerical stability. The first TT core $\ten{G}_1$ is given by the $3\times 3$ matrix $U_1$ indexed by $i_1, \alpha_1$. We now form the matrix $Y_1=S_1V_1^T$ and reshape it such that its rows are indexed by $\alpha_1, i_2$ and its columns by $i_3$. This results in a $12\times 2$ matrix $\bar{Y}_1$ and its SVD
$$
\bar{Y}_1 = U_2\; Y_2,
$$ 
with $Y_2=S_2V_2^T$. The second TT core $\ten{G}_2$ is then given by $U_2$, reshaped into a $3\times 4 \times 2$ tensor. The last TT core $\ten{G}_3$ is then $Y_2$, which is a $2\times 2$ matrix indexed by $\alpha_2$ and $i_3$. We therefore have that
$$
\ten{A}_{i_1 i_2 i_3} = \ten{G}_1(i_1) \; \ten{G}_2(i_2) \; \ten{G}_3(i_3),
$$
with $\ten{G}_1(i_1)$ a $1 \times 3$ row vector, $\ten{G}_2(i_2)$ a $3\times 2$ matrix, and $\ten{G}_3(i_3)$ a $2 \times 1$ column vector for fixed $i_1,i_2$ and $i_3$, respectively (cf. Figure \ref{fig:TTA}). Observe how the auxiliary indices $\alpha_1,\alpha_2$ serve as `links' connecting adjacent TT cores. For tensor orders $d\geq3$, besides the head and tail tensors $\ten{G}_1(i_1),\ten{G}_d(i_d)$ which are in fact matrices, there will be $(d-2)$ third-order TT cores in between.

\subsection{Tensor-Train rank-1 decomposition}
\label{sec:ttr1}
With the TT decomposition in place, we are now ready to introduce our TTr1 decomposition, which is easily understood from Figure \ref{fig:TTr1A}. The main idea of the TTr1 decomposition is to force the rank for each auxiliary index $\alpha_k$ link to unity, which gives rise to a linear combination of rank-1 outer products. We go back to the first SVD \eqref{ex:SVD1} of the TT decomposition algorithm and realize that we can rewrite it as a sum of rank-1 terms
\begin{equation}
\bar{\ten{A}} \;=\; \sum_{i=1}^{3} \, \sigma_i {_{\times_1}} u_i {_{\times_2}} v_i,
\label{eq:SVDsum1terms}
\end{equation}
where each vector $v_i$ is indexed by $i_2, i_3$. The next step in the TT decomposition would be to reshape $Y_1$ and compute its SVD. For the TTr1 decomposition we reshape each $v_i$ into an $i_2 \times i_3$ matrix $\bar{v}_i$ and compute its SVD. This allows us to write $\bar{v}_1$ also as a sum of rank-1 terms
$$
\bar{v}_1 \;= \; \sigma_{11} {_{\times_1}} u_{11} {_{\times_2}} v_{11} + \sigma_{12} {_{\times_1}} u_{12} {_{\times_2}} v_{12}.
$$
The same procedure can be done for $v_2$ and $v_3$: they can also be written as a sum of 2 rank-1 terms. Combining these 6 rank-1 terms we can finally write $\ten{A}$ as
\begin{align}
\label{ex:Xttr1}
\ten{A}  &=  \tilde{\sigma}_1  {_{\times_1}} u_{1} {_{\times_2}} u_{11} {_{\times_3}} v_{11} + \tilde{\sigma}_2{_{\times_1}} u_{1} {_{\times_2}} u_{12} {_{\times_3}} v_{12}  \\
&+ \tilde{\sigma}_3 {_{\times_1}} u_{2} {_{\times_2}} u_{21} {_{\times_3}} v_{21} +\tilde{\sigma}_4 {_{\times_1}} u_{2} {_{\times_2}} u_{22} {_{\times_3}} v_{22}  \nonumber \\
 & + \tilde{\sigma}_5 {_{\times_1}} u_{3} {_{\times_2}} u_{31} {_{\times_3}} v_{31} + \tilde{\sigma}_6 {_{\times_1}} u_{3} {_{\times_2}} u_{32} {_{\times_3}} v_{32} \nonumber ,
\end{align}
with $\tilde{\sigma}_1=\sigma_1\sigma_{11},\ldots,\tilde{\sigma}_6=\sigma_3\sigma_{32}$. Note the similarity of \eqref{ex:Xttr1} with \eqref{eq:SVDsum1terms}. The TTr1 decomposition has three main features that render it similar to the matrix SVD:
\begin{enumerate}
\item the scalars $\tilde{\sigma}_1,\ldots,\tilde{\sigma}_6$ are the weights of the outer products in the decomposition and can therefore be thought of as the singular values of $\ten{A}$,
\item the outer products affiliated with each singular value are tensors of unit Frobenius norm, since each product vector (or mode vector) is a unit vector, and
\item each outer product in the decomposition is orthogonal to all the others, which we will prove in Section \ref{sec:properties}.
\end{enumerate}
\begin{figure}[ht]
\centering
\includegraphics[width=.95\textwidth]{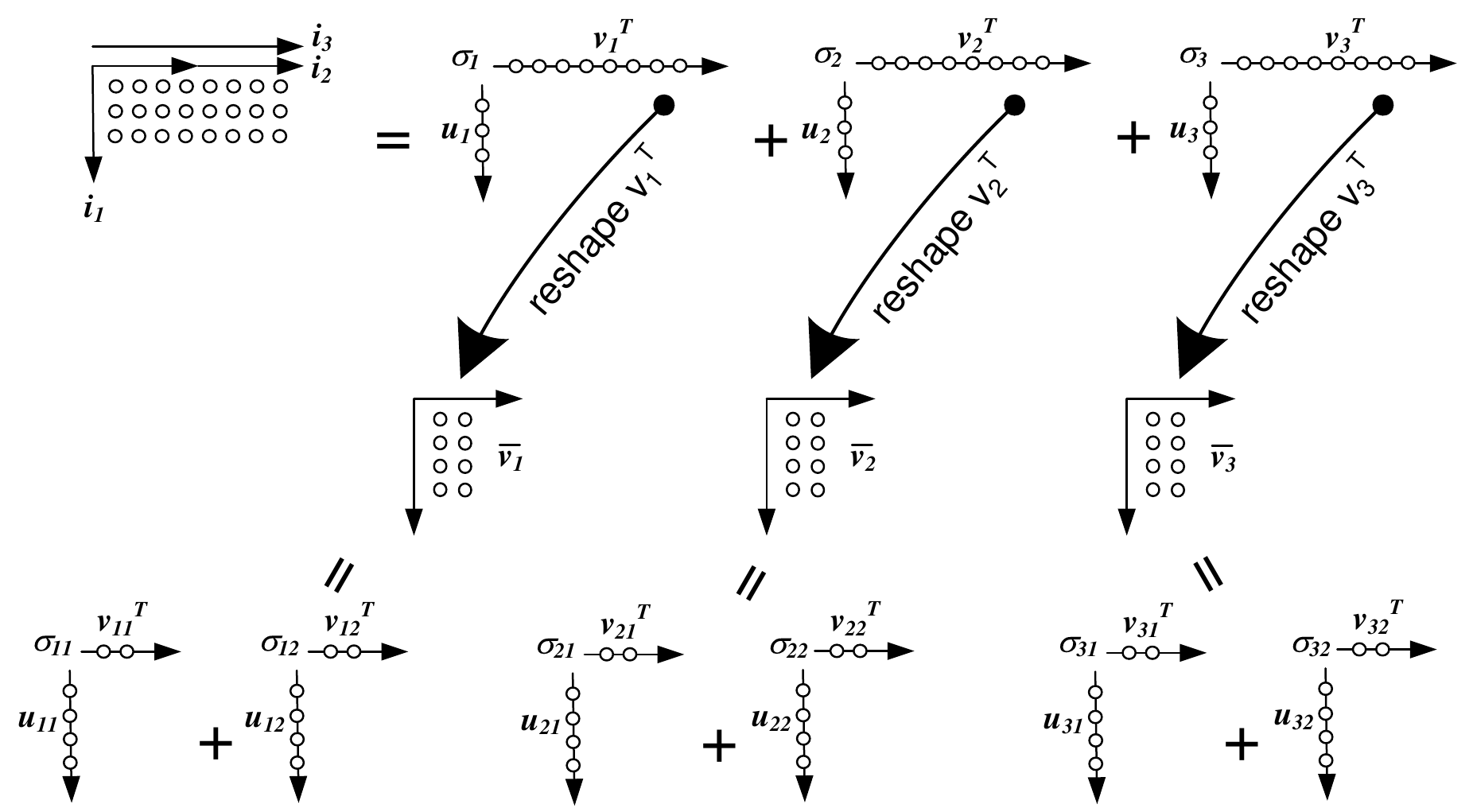}
\caption{{Computation of the TTr1 decomposition of $\ten{A}$.}}
\label{fig:TTr1A}
\end{figure}

\subsection{TTr1SVD algorithm}
As was shown in the previous subsection, computing the TTr1 decomposition requires recursively reshaping the obtained $v$ vectors and computing their SVDs. This recursive procedure gives rise to the formation of a tree, where each SVD generates additional branches of the tree. The tree for the TTr1 decomposition of $\ten{A}$ in~\eqref{ex:Xttr1} is shown in Figure \ref{fig:tree}. As denoted in the figure, we will call a row in the tree a level. Level 0 corresponds to the SVD of $\bar{\ten{A}}$ and generates the first level of singular values. This is graphically represented by the node at level 0 branching off into 3 additional nodes at level 1. The reshaping and SVD of the different $v$ vectors at level 1 then generates level 2 and so forth. Observe how the total number of subscript indices of the singular values are equal to the level at which these singular values occur. For example, $\sigma_2$ occurs at level 1 and $\sigma_{21}$ occurs at level 2. The number of levels for the TTr1 decomposition of an arbitrary $d$-way tensor is $d-1$. The final singular values $\tilde{\sigma}_i$'s are the product of all $\sigma$'s along a branch.
\begin{figure}[th]
\begin{center}
\input{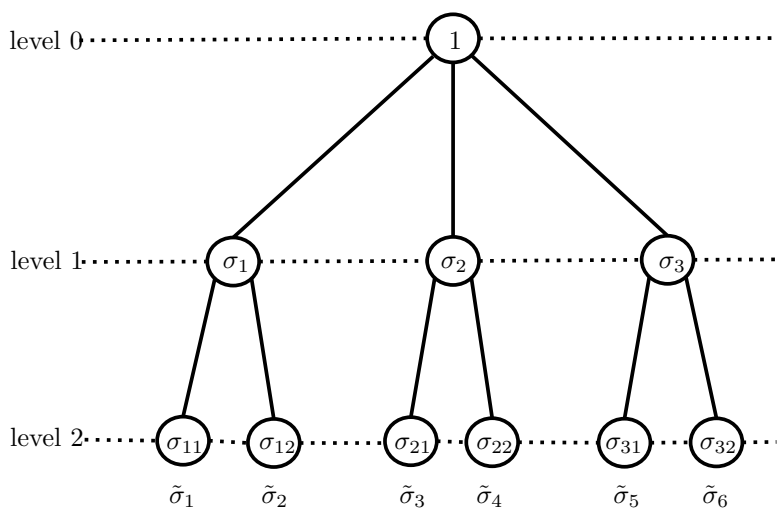}
\caption{Tree representation of TTr1 decomposition, where $\tilde{\sigma}_i$ is the product of all nodes down a branch.}
\label{fig:tree}
\end{center}
\end{figure}

The total number of terms in the decomposition are the total number of leaves. This number is easily determined. Indeed, each node at level $k$ of the tree branches off into
$$
r_k \;\triangleq \; \textrm{min}(n_{k+1},\prod_{i=k+2}^d n_i) \quad (k=0,\ldots,d-2),
$$
nodes. Hence, the total number of leaves or terms $N$ in the TTr1 decomposition is given by
$$
N\;=\;\prod_{k=0}^{d-2}\,r_k.
$$
The algorithm to compute the TTr1 decomposition is presented in pseudo-code in Algorithm \ref{alg:ttr1svd}. First the tensor $\ten{A}$ is reshaped into an $n_1\times \prod_{i=2}^d n_i$ matrix and its SVD is computed. The computational complexity for this first step is approximately $14\,n_1^2\, \prod_{i=2}^d n_i+8\,n_1^3$ flops. Observe that the computation of the TT or Tucker decomposition has a computational complexity of the same order of magnitude. Then for all remaining nodes in the tree, except for the leaves, the resulting $v_i$ vectors are reshaped into a matrix and their SVDs are also computed. The $U,S,V$ matrices for each of these SVDs are stored. Note that for levels 0 up to $d-2$ the $v$ vectors do not need to be stored. From the tree it is also easy to determine the total number of SVDs required to do the full TTr1 decomposition. This is simply the total number of nodes in the tree from level 0 up to $d-2$ and equals
$$
1 + \sum_{i=0}^{d-3} \, \prod_{k=0}^i r_k .
$$
Assuming that $r_k=n_k$ for all $k$, then the total number of SVDs required for computing the TTr1 decomposition of a cubical tensor is
$$
1+n+n^2+\cdots+n^{d-2} \;=\; \frac{1-n^{d-1}}{1-n}.
$$
This exponential dependence on the order of the tensor and the computational complexity of $O(n_1^2\, \prod_{i=2}^d n_i)$ for the first SVD are the two major limiting factors to compute the TTr1 decomposition. Note, however, that the tree structure is perfectly suited to do all SVD computations that generate the next level in parallel, and in that case the runtime is linearly proportional to the number of levels. However, such an implementation requires an exponential growing number of computational units.\\
\\
\framebox[.9\textwidth][l]{\begin{minipage}{0.9\textwidth}
\begin{algorithm}Tensor-Train rank-1 SVD Algorithm  (TTr1SVD)\\
\textit{\textbf{Input}}: arbitrary tensor $\ten{A}$\\
\textit{\textbf{Output}}: $U,S,V$ matrices of each SVD
\begin{algorithmic}
\STATE $\bar{\ten{A}} \gets$ reshape $\ten{A}$ into an $n_1\times \prod_{i=2}^d n_i$ matrix
\STATE $U_1,S_1,V_1 \gets$ SVD($\bar{\ten{A}}$)
\FOR{all remaining nodes in the tree except the leaves}
\STATE $\bar{v}_i \gets$ reshape $v_i$
\STATE $U_k,S_k,V_k \gets$ SVD($\bar{v}_i$)
\STATE add $U_k,S_k,V_k$ to $U,S,V$
\ENDFOR
\end{algorithmic}
\label{alg:ttr1svd}
\end{algorithm}
\end{minipage}
}

\section{Properties}
\label{sec:properties}
We now discuss many attractive properties of the TTr1 decomposition. Most of these properties are also shared with the matrix SVD and it is in this sense that the TTr1SVD is a natural generalization of the SVD for tensors.

\subsection{Uniqueness}
A first attractive feature of the TTr1 decomposition is that it is uniquely determined for a fixed order of indices. This means that for any given arbitrary tensor $\ten{A}$ its TTr1 decomposition will always be the same. Indeed, Algorithm \ref{alg:ttr1svd} consists of a sequence of SVD computations so the uniqueness of the TTr1 decomposition follows trivially from the fact that each of the SVDs in Algorithm \ref{alg:ttr1svd} are unique up to sign. Although the singular values and vectors of a matrix $A$ and its transpose $A^T$ are the same, this is not the case for the TTr1SVD. Indeed, applying a permutation of the indices $\pi(i_1,\ldots,i_n)$ will generally result in a different TTr1SVD, which we illustrate in Example \ref{ex:ex1}. Once the indices are fixed however, the TTr1SVD algorithm will always return the same decomposition, which is not the case for conventional iterative optimization-based methods.

\subsection{Orthogonality of outer products}
Any two rank-1 terms $\tilde{\sigma}_i \ten{T}_i$ and $\tilde{\sigma}_j \ten{T}_j$ of the TTr1 decomposition are orthogonal with respect to one another, which means that $\langle \ten{T}_i, \ten{T}_j \rangle=0$. We will use our running example to show why this is so. Let us take two terms of \eqref{ex:Xttr1}, for example $\ten{T}_1=1{_{\times_1}} u_{1} {_{\times_2}} u_{11} {_{\times_3}} v_{11}$ and $\ten{T}_2=1 {_{\times_1}} u_{1} {_{\times_2}} u_{12} {_{\times_3}} v_{12}$. Another way of writing $\langle \ten{T}_1, \ten{T}_2 \rangle$ is
$$
\langle \ten{T}_1, \ten{T}_2 \rangle \;=\; \left( v_{11} \otimes u_{11} \otimes u_1 \right)^T \left(v_{12} \otimes u_{12} \otimes u_1 \right)
$$
where $\otimes$ denotes the Kronecker product. These Kronecker products generate the vectorization of each of the rank-1 tensors $\ten{T}_1, \ten{T}_2$, which allows us to easily write down their inner product as an inner product between two mode vectors. Applying properties of the Kronecker product we can now write
\begin{align*}
\left(v_{11} \otimes u_{11} \otimes u_1 \right)^T \left(v_{12} \otimes u_{12} \otimes u_1 \right) &= \left(v_{11}^T \otimes u_{11}^T \otimes u_1^T \right) \left(v_{12} \otimes u_{12} \otimes u_1 \right),\\
 &= \left(v_{11}^Tv_{12}  \otimes u_{11}^T u_{12} \otimes u_1^Tu_1 \right),
\end{align*}
where it is clear that the right hand side vanishes due to the orthogonality $v_{11}^T v_{12}=u_{11}^T u_{12}=0$. This property generalizes to any tensor $\ten{A}$. Indeed, if any two rank-1 terms do not originate from the same node at level 1, then their respective $u_i,u_j$ vectors are orthogonal and ensure that their inner product vanishes. If the two rank-1 terms do originate from the same node at level 1 but from different nodes at level 2, then their $u_{ij},u_{ik}$ vectors are orthogonal and again the inner product will vanish. This reasoning extends up to level $d-1$. If any two terms have their first separate nodes at level $k\in [1,d-1]$, then their corresponding $u$ vectors at level $k$ will also be orthogonal. The tree structure, together with the orthogonality of all $u$ vectors that share a same parent node hence guarantees that any two rank-1 outer factors in the TTr1 decomposition are orthogonal. The TTr1 decomposition is hence an orthogonal decomposition as defined in~\cite{Kolda2001}. 

\subsection{Upper bound on the orthogonal tensor rank}
\label{subsec:upperbound}
The (CP-)rank of an arbitrary $d$-way tensor $\ten{A}$ is usually defined similarly to the matrix case as the minimum number of rank-1 terms that $\ten{A}$ decomposes into.
\begin{definition}
The rank of an arbitrary $d$-way tensor $\ten{A}$, denoted $\textrm{rank}(\ten{A})$, is the minimum number of rank-1 tensors that yield $\ten{A}$ in a linear combination.
\end{definition}

In \cite{Kolda2001,Leibovici1998} the orthogonal rank, rank$_\perp(\ten{A})$, is defined as the minimal number of terms in an orthogonal rank-1 decomposition. Apparently,
$$
\textrm{rank}(\ten{A}) \leq \textrm{rank}_\perp(\ten{A})
$$
where strict inequality is possible for tensors of orders $d > 2$.
The TTr1 decomposition allows a straightforward determination of an upper bound on rank$_\perp(\ten{A})$. Indeed, this is simply the total number of leaves in the tree and is therefore
\begin{equation}
\textrm{rank}_\perp(\ten{A}) \; \leq \; N\;=\;\prod_{k=0}^{d-2}\,r_k.
\label{eq:upperboundrank}
\end{equation}
Applying \eqref{eq:upperboundrank} to our running example $\ten{A} \in \mathbb{R}^{3\times 4 \times 2}$ we obtain
$$
\textrm{rank}_\perp(\ten{A}) \;\leq \; \textrm{min}(3,8) \cdot \textrm{min}(4,2) \;=\; 3 \cdot 2 \;=\; 6.
$$
For a cubical tensor with $n_1=\cdots=n_d=n$,~\eqref{eq:upperboundrank} then tells us that
$$
\textrm{rank}_\perp(\ten{A}) \; \leq \; \prod_{k=0}^{d-2} \textrm{min}(n, n^{d-1}) = \prod_{k=0}^{d-2}n = n^{d-1} .
$$
The dependency of the TTr1SVD on the ordering of the indices implies that a permutation of the indices can lead to different upper bounds on the orthogonal rank. Indeed, if we permute the indices of $\ten{A}$ to $\{i_2,i_3,i_1\}$ we get 
$$
\textrm{rank}_\perp(\ten{A})\;\leq \; \textrm{min}(4,6) \cdot \textrm{min}(2,3) \;=\; 4 \cdot 2 \;=\; 8.
$$
Consequently, there exists the notion of a minimum upper bound on the orthogonal rank of a tensor, obtained from computing the rank upper bounds through all permutations of indices. Whether the TTr1SVD algorithm is able to derive a minimal orthogonal decomposition needs further investigation. Furthermore, we will demonstrate by an example in Section \ref{sec:examples} that the orthogonality as it occurs in the TTr1 decomposition is not enough to make the problem of computing a low-rank approximation of an arbitrary tensor well-posed. This agrees with \cite{Vannieuwenhoven2014}, in which a necessary condition of pairwise orthogonality of all rank-1 terms in at least 2 modes is proved.

\subsection{Quantifying the approximation error}
As soon as the number of levels is large it becomes very cumbersome to write all the different subscript indices of the $u$ and $v$ vectors in the TTr1 decomposition. We therefore introduce a shorter and more convenient notation. Herein, $u_{ki}$ denotes the $u$ vector at level $k$ that contributes to the $i$th rank-1 term. Similarly, $v_i$ denotes the $v$ vector that contributes to the $i$th rank-1 term. The TTr1SVD algorithm decomposes an arbitrary tensor $\ten{A}$ into a linear combination of $N$ orthogonal rank-1 terms
\begin{equation}
\ten{A} \;=\; \sum_{i=1}^N \tilde{\sigma}_i {_{\times_1}} u_{1i} {_{\times_2}} u_{2i} {_{\times_3}} \cdots {_{\times_{d-1}}} u_{d-1 i} {_{\times_d}} v_{i},
\label{eq:ttr1}
\end{equation}
with
$$
||1 {_{\times_1}} u_{1i} {_{\times_2}} u_{2i} {_{\times_3}} \cdots {_{\times_{d-1}}} u_{d-1 i} {_{\times_d}} v_{i} ||_F = 1 \quad \textrm{ and } \quad 
N = \prod_{k=0}^{d-2} r_k.
$$
Suppose that we have ordered and relabeled the terms such that $\tilde{\sigma}_1 \geq \tilde{\sigma}_2 \geq \cdots \geq \tilde{\sigma}_N$. An $R$-term approximation is then computed by truncating \eqref{eq:ttr1} to the first $R$ terms
$$
\tilde{\ten{A}} \;=\;  \sum_{i=1}^R \tilde{\sigma}_i {_{\times_1}} u_{1i} {_{\times_2}} u_{2i} {_{\times_3}} \cdots {_{\times_{d-1}}} u_{d-1 i} {_{\times_d}} v_{i}.
$$
The following lemma tells us exactly what the error is when breaking off the summation at $R$ terms.
\begin{lemma}
Let $\tilde{\ten{A}}$ be the summation of the first $R$ terms in \eqref{eq:ttr1} then
$$
||\ten{A}-\tilde{\ten{A}}||_F\;=\;\sqrt{\tilde{\sigma}_{R+1}^2 + \cdots + \tilde{\sigma}_N^2}.
$$
\label{lem:errorlem}
\end{lemma}
\begin{proof}
Using the fact that $||1 {_{\times_1}} u_{1i} {_{\times_2}} u_{2i} {_{\times_3}} \cdots {_{\times_{d-1}}} u_{d-1 i} {_{\times_d}} v_{i}||_F = 1$ we can write
$$
||\ten{A}-\tilde{\ten{A}}||_F = ||\sum_{i=R+1}^n \tilde{\sigma}_i {_{\times_1}} u_{1i} {_{\times_2}} u_{2i} {_{\times_3}} \cdots {_{\times_{d-1}}} u_{d-1 i} {_{\times_d}} v_{i} ||_F= \sqrt{\tilde{\sigma}_{R+1}^2 + \cdots + \tilde{\sigma}_N^2}.
$$ 
\quad \end{proof}

Lemma \ref{lem:errorlem} can also be used to determine the lowest number of terms $R$ with a guaranteed accuracy. Indeed, once a tolerance $\epsilon$ is chosen such that it is required that
$$
||\ten{A}-\tilde{\ten{A}}||_F\;\leq \; \epsilon,
$$
the minimal number of terms $R$ in the TTr1 decomposition of $\tilde{\ten{A}}$ is easily determined by the requirement that
$$
\sqrt{\tilde{\sigma}_{R+1}^2 + \cdots + \tilde{\sigma}_N^2} \leq \epsilon \textrm{ such that } \sqrt{\tilde{\sigma}_{R}^2 + \cdots + \tilde{\sigma}_N^2} > \epsilon.
$$
It is tempting to choose $R$ such that $\tilde{\sigma}_R > \epsilon > \tilde{\sigma}_{R+1}$. However, when the approxi-rank gap, defined as $\tilde{\sigma}_R/\tilde{\sigma}_{R+1}$ \cite[p.~920]{Zeng2005}, is not large enough then there is a possibility that $\sqrt{\tilde{\sigma}_{R+1}^2 + \cdots + \tilde{\sigma}_N^2} \geq \epsilon$ due to the contributions of the smaller singular values. A large approxi-rank gap implies that the number of terms in the approximation $\tilde{\ten{A}}$ is relatively insensitive to the given tolerance $\epsilon$. In Example 6 a tensor is presented for which this is not the case.

\subsection{Reducing the number of SVDs}
Suppose that an approximation $\tilde{\ten{A}}$ of $\ten{A}$ is desired such that $||\ten{A}-\tilde{\ten{A}}||_F\;\leq \; \epsilon$. Computing the full TTr1 decomposition and applying Lemma \ref{lem:errorlem} solves this problem. It is, however, possible to reduce the total number of required SVDs by taking into account that the final singular values $\tilde{\sigma}_i$'s are the product of the singular values along each branch of the TTr1-tree. An important observation is that all singular values $\sigma_{ij\cdots m}$ at levels 2 up to $d-1$ satisfy $\sigma_{ij\cdots m} \leq 1$. This is easily seen from the fact that they are computed from a reshaped unit vector $v_{ij\cdots n}$ at their parent node. Indeed, since $||v_{ij\cdots n}||_2=1$ it follows that $||\bar{v}_{ij\cdots n}||_F=1$. This allows us to make an educated guess at the impact of the singular values at level $l$ on the final rank-1 terms. Suppose we have a singular value $\sigma_k$ at level $l$, preceded by a product $\sigma_{k0}$ of parent singular values. An upper bound on the size of the final $\tilde{\sigma}$'s that are descendants from $\sigma_k$ can be derived by assuming that $\sigma_k$ is unchanged throughout each branch. Since one node at level $l$ results in $\prod_{i=l}^{d-2} r_i$ rank-1 terms, this then implies that there are $\prod_{i=l}^{d-2} r_i$ rank-1 terms with $\tilde{\sigma} = \sigma_{k0}\sigma_k^{d-l-1}$, so
$$
\tilde{\sigma}_{1}^2 + \cdots + \tilde{\sigma}_N^2 \leq \tilde{\sigma}_{1}^2 + \cdots + \tilde{\sigma}_r^2 + \prod_{i=l}^{d-2} r_i (\sigma_{k0} \sigma_k^{d-l-1})^2.
$$
If now 
\begin{equation}
e_k^2 \triangleq \prod_{i=l}^{d-2} r_i (\sigma_{k0} \sigma_k^{d-l-1})^2 \leq \epsilon^2
\label{eq:reducecond}
\end{equation}
is satisfied then removing $\sigma_k$ at level $l$ produces an approximation $\tilde{\ten{A}}$ that is guaranteed to satisfy the approximation error bound. Removing $\sigma_k$ at level $l$ implies that not a full but a reduced TTr1 decomposition is computed. Indeed, the total number of computed rank-1 terms is effectively lowered by $\prod_{i=l}^{d-2} r_i$ terms, decreasing the total number of required SVDs in the TTr1SVD algorithm. This condition on $\sigma_k$ is easily extended to $m$ singular values at level $l$ as
\begin{equation}
\sum_{j=1}^m e_j^2 \leq \epsilon^2,
\label{eq:prunecondition}
\end{equation}
where we compute an $e_j$ term for each of the $m$ singular values at level $l$. Checking whether \eqref{eq:prunecondition} holds for $m$ $\sigma$'s at level $l$ can be easily implemented in Algorithm \ref{alg:ttr1svd}. As shown in Section \ref{sec:examples}, a rather gradual decrease of $\sigma_k$ is seen in practice as the level increases. This implies that it might still be possible to find a $\tilde{\ten{A}}$ of lower rank that satisfies the approximation error bound from the rank-1 terms of a reduced TTr1 decomposition. Lemma \ref{lem:errorlem} can also be used to find the desired $\tilde{\ten{A}}$ in this case.

\subsection{Orthogonal complement tensors}
\label{subsec:Xperp}
We can consider the vectorization of $\ten{A}$ as a vector living in an $(n_1\cdots n_d)$-dimensional vector space. Naturally, there must be a $(n_1\cdots n_d-1)$-dimensional vector space $\spn(\ten{A})^\perp$ of tensors that are orthogonal to $\ten{A}$. Note that each basis vector of $\spn(\ten{A})^\perp$ is required to be the vectorization of an outer product of vectors. The TTr1 decomposition allows us to easily find an orthogonal basis for $\spn(\ten{A})^\perp$. We will illustrate how this comes about using the tensor $\ten{A}$ from Figure~\ref{fig:A} and notions in~\eqref{ex:Xttr1}. Recall from Section~\ref{sec:ttr1} that the first step in the TTr1SVD algorithm was the economical SVD of the $3\times 8$ matrix $\bar{\ten{A}}=USV^T$. Each of the $v_i$ vectors was then reshaped into a $4\times 2$ matrix $\bar{v}_i$. Now consider a full SVD of each of these $\bar{v}_i$ matrices
\begin{equation}
\bar{v}_i \;=\; \begin{pmatrix}
u_{i1} & u_{i2} & u_{i3} & u_{i4}
\end{pmatrix} \, 
\begin{pmatrix}
\sigma_{i1} & 0\\
0 & \sigma_{i2}\\
0 & 0 \\
0 & 0
\end{pmatrix}
\, \begin{pmatrix}
v_{i1}^T\\
v_{i2}^T \\
\end{pmatrix},
\label{eq:fullsvd}
\end{equation}
which is a sum of 8 orthogonal rank-1 outer products with only 2 nonzero $\sigma$'s. There are hence 6 additional outer product terms $u_i \circ u_{ij} \circ v_{ij}$ with a zero singular value, orthogonal to the outer product terms of the economical TTr1 decomposition \eqref{ex:Xttr1}. It is easily seen that the rank-1 terms obtained from the zero entries of the $S$ matrix in \eqref{eq:fullsvd} are orthogonal to $\ten{A}$ and are therefore basis vectors of $\spn(\ten{A})^\perp$. Table \ref{tab:u1terms} lists all 8 orthogonal rank-1 outer product terms that are obtained for the $u_1$ branch in the TTr1-tree. Each rank-1 term can be read off from Table \ref{tab:u1terms} by starting from the top row and going down along a particular branch of the TTr1-tree. For example, the fourth rank-1 term is given by $\tilde{\sigma}_2\, u_1 \circ u_{12} \circ v_{12}$ and the seventh rank-1 term by $0\, u_1 \circ u_{14} \circ v_{11}$. We call such a table that exhibits the TTr1-tree structure and allows us to reconstruct all rank-1 terms an outer product column table. The extra orthogonal terms for the $u_2,u_3$ branches are completely analogous to the $u_1$ branch. Note that the economical TTr1 decomposition described in Section \ref{sec:ttr1} only computes the $\tilde{\sigma}_1 \, u_1 \circ u_{11} \circ v_{11}$ and $\tilde{\sigma}_2 \, u_1 \circ u_{12} \circ v_{12}$ terms.
\begin{table}[ht]
\begin{center}
\caption{All 8 orthogonal rank-1 outer products in the $u_1$ branch of the TTr1-tree.}
\label{tab:u1terms}

		\begin{tabular}{|c|c|c|c|c|c|c|c|}
		\hline

		\multicolumn{8}{|c|}{$u_1$}\\
\hline
		\multicolumn{2}{|c|}{$u_{11}$} & \multicolumn{2}{|c|}{$u_{12}$} & \multicolumn{2}{|c|}{$u_{13}$} & \multicolumn{2}{|c|}{$u_{14}$} \\
		\hline
		{$v_{11}$} & { $v_{12}$} & { $v_{11}$} & { $v_{12}$} & {$v_{11}$} & {  $v_{12}$} & {$v_{11}$} & {$v_{12}$}\\
	\hline
		{$\tilde{\sigma}_{1}$} & {$0$} & {$0$} & {$\tilde{\sigma}_{2}$} & {$0$} & {$0$} & {$0$} & {$0$}\\
	\hline
		\end{tabular}
	\end{center}
\end{table}


The full TTr1 decomposition therefore consists of $8 \times 3 = 24$ orthogonal terms and can be written in vectorized form as
$$
\textrm{vec}(\ten{A}) \;=\; \begin{pmatrix}\textrm{vec}(\ten{T}_1) & \cdots & \textrm{vec}(\ten{T}_6) & \textrm{vec}(\ten{T}_7) & \cdots \textrm{vec}(\ten{T}_{24}) \end{pmatrix} \begin{pmatrix}
\tilde{\sigma}_1 \\
\vdots\\
\tilde{\sigma}_6 \\
0\\
\vdots\\
0
\end{pmatrix}
$$
where $\textrm{vec}(\ten{T}_1),\ldots ,\textrm{vec}(\ten{T}_6)$ are the orthogonal terms computed in Section \ref{sec:ttr1} and $\textrm{vec}(\ten{T}_7),\ldots ,\textrm{vec}(\ten{T}_{24})$ are the orthogonal terms that partly span $\spn(\ten{A})^\perp$. Note that we have only found 18 basis vectors for $\spn(\ten{A})^\perp$. The remaining 5 basis vectors are to be found as the following linear combinations of $\textrm{vec}(\ten{T}_1),\ldots ,\textrm{vec}(\ten{T}_{6})$
$$
\begin{pmatrix}\textrm{vec}(\ten{T}_1) & \cdots & \textrm{vec}(\ten{T}_6)\end{pmatrix} S^\perp,
$$
where $S^\perp$ is the $6\times 5$ matrix orthogonal to $\begin{pmatrix} \tilde{\sigma}_1 & \cdots & \tilde{\sigma}_6 \end{pmatrix}$. The property that for every tensor $\ten{B} \in \spn(\ten{A})^\perp$ we have that $\langle \ten{A},\ten{B} \rangle = 0$, allows us to interpret $\spn(\ten{A})^\perp$ as the orthogonal complement of $\textrm{vec}(\ten{A})$.

\subsection{Constructive proof maximal CP-rank of $2 \times 2 \times 2$ tensor}
As an application of the outer product column table, we show how it leads to an elegant proof of the maximal CP-rank of a real $2\times 2\times 2$ tensor over $\mathbb{R}$. It is known that the maximum rank of a real $2\times 2\times 2$ tensor over $\mathbb{R}$ is 3 (i.e., any such tensor can be expressed as the sum of at most 3 real outer products~\cite{tensorreview}), for which rather complicated proofs were given in~\cite{Kruskal1989,tenberge1991}. Incidentally, we show that the TTr1 decomposition allows us to formulate a remarkably simpler proof. As in Section \ref{subsec:Xperp}, we first consider all orthogonal outer products that span $\mathbb{R}^{2\times 2\times 2}$ in the outer product Table~\ref{tab:column}.
\begin{table}[ht]
	\begin{center}
\caption{Outer product table for a general $2\times 2\times 2$ tensor.}
	\label{tab:column}	
		\begin{tabular}{|c|c|c|c|c|c|c|c|}
		\hline
		\multicolumn{4}{|c|}{$u_1$} & \multicolumn{4}{|c|}{$u_2$} \\
		 \hline
		\multicolumn{2}{|c|}{$u_{11}$} & \multicolumn{2}{|c|}{$u_{12}$} & \multicolumn{2}{|c|}{$u_{21}$} & \multicolumn{2}{|c|}{$u_{22}$} \\
	 \hline
		{ $v_{11}$} & {$v_{12}$} & {$v_{11}$} & {$v_{12}$} & {$v_{21}$} & {$v_{22}$} & {$v_{21}$} & {$v_{22}$}\\
	 \hline
		{$\tilde{\sigma}_1$} & {$0$ } & {$0$ } & {$\tilde{\sigma}_2$ } & {$\tilde{\sigma}_3$} & {$0$} & {$0$} & {$\tilde{\sigma}_4$}\\ \hline	
		\end{tabular}
	\end{center}

\end{table}

The columns in Table~\ref{tab:column} with nonzero singular values are the ``active'' columns in the TTr1 decomposition of a random real tensor $\ten{A} \in\mathbb{R}^{2\times 2\times 2}$. The ``inactive'' (orthogonal) columns carry zero weights but are crucial for proving the maximum rank-3 property of $\ten{A}$. 

We first enumerate two important yet straightforward properties for the columns in Table~\ref{tab:column} ignoring the bottom row for the time being. First, scaling a column can be regarded as multiplying a scalar onto the whole outer product or absorbing it into any one of the mode vectors. Taking the first column and a scalar $\alpha\in\mathbb{R}$ for instance, this means that%
\begin{align*}
\alpha (u_1\circ u_{11} \circ v_{11}) = (\alpha u_1) \circ u_{11} \circ v_{11} = u_1 \circ (\alpha u_{11}) \circ v_{11} = u_1 \circ  u_{11} \circ (\alpha v_{11}).
\end{align*}
In other words, the scalar is ``mobile'' across the various modes. The second property is that any two columns differing in only one mode can be added to form a new rank-1 outer product~\cite{Kolda2001}. We list two examples showing the rank-1 outer products resulting from the linear combinations of columns 1 and 3, and columns 3 and 4, respectively,%
\begin{align*}
\alpha (u_1\circ u_{11} \circ v_{11}) + \beta  (u_1\circ u_{12} \circ v_{11})&= u_1 \circ (\alpha u_{11} + \beta u_{12}) \circ v_{11},\\
\alpha (u_1\circ u_{12} \circ v_{11}) + \beta  (u_1\circ u_{12} \circ v_{12})&= u_1 \circ  u_{12} \circ (\alpha v_{11}+\beta v_{12}).%
\end{align*}

Now to prove the maximum rank-3 property of $\ten{A}$ then, is to show that the four active columns of the outer product column table~\ref{tab:column} can always be ``merged'' into three. To begin with, it is readily seen that if we add any nonzero multiple of column 3 to column 1, and then subtract the same multiple of column 3 from column 4, the overall tensor by summing all columns in Table~\ref{tab:column} remains unchanged. Our final goal is to merge columns 1 and 5 into one outer product by making two of their modes the same (up to a scalar factor). This is done by appropriately adding column 3 to column 1 such that the second mode vectors of columns 1 and 5 align, while adding column 6 to column 5 such that the third mode vectors of columns 1 and 5 align. Of course, subtractions of column 3 from column 4, and column 6 from column 8, respectively, are necessary to offset the addition. This intermediate step is summarized in Table~\ref{tab:columnProof} wherein the four intermediate columns are now shown individually with the $\tilde{\sigma}_i$'s absorbed into the mode vectors.

\begin{table}[ht]
\centering
\begin{center}
\caption{The four intermediate outer products that can be merged into three outer products.}
\label{tab:columnProof}
\begin{tabular}{@{}cccc@{}}
{$u_1$} &  {$u_1$} & {$u_2$} & {$u_2$}\\ \midrule
{$\tilde{\sigma}_1 u_{11}$} &  {} & {} & {$\tilde{\sigma}_4 u_{22}$}\\
{$+\alpha u_{12}$} & {$u_{11}$} & {$u_{21}$} & {$-\gamma u_{21}$}\\
{$(=\beta u_{21})$} & {} & {} & {}\\\midrule
{} & {$-\alpha v_{11} $} & {$\tilde{\sigma}_3 v_{21}$} & {}\\
{$v_{11}$} & {$+ \tilde{\sigma}_2 v_{12}$} & {$+\gamma v_{22}$} & {$v_{22}$}\\		
{} & {} & {$(=\delta v_{11})$} & {}\\
		\end{tabular}
	\end{center}
\end{table}

The two linear equations that need to be solved in this process are%
\begin{align*}
\left[ {\begin{array}{*{20}c}
   { - u_{12} } & {u_{21} }  \\
 \end{array} } \right]\left[ {\begin{array}{*{20}c}
   \alpha   \\
   \beta   \\
 \end{array} } \right] = \tilde{\sigma}_1 u_{11}~~\mbox{and}~~\left[ {\begin{array}{*{20}c}
   { - v_{22} } & {v_{11} }  \\
 \end{array} } \right]\left[ {\begin{array}{*{20}c}
   \gamma   \\
   \delta   \\
 \end{array} } \right] = \tilde{\sigma}_3 v_{21}.%
\end{align*}
It is not hard to see that columns 1 and 3 of Table \ref{tab:columnProof} can now be merged into one outer product
$$
 (\beta u_1 + \delta u_2) \circ u_{21} \circ v_{11}
$$
due to two of their mode vectors now being parallel. Hence an overall rank-3 representation for the original tensor $\ten{A}$ is obtained from its TTr1 decomposition.

Obviously, this rank-3 representation is not unique since alternatively we can first align the third mode of columns 1 and 5, followed by their second mode. Furthermore, instead of columns 1 and 3, we can also merge columns 2 and 4 etc.. Details are omitted as they are all based on the same idea of merging columns. Another big advantage of our rank-3 construction is that the relative numerical error $||\ten{A}-\tilde{\ten{A}}||_F / ||\ten{A}||_F < 10^{-15}$, whereas the CP rank-3 decomposition has a median error of $\approx 10^{-6}$ over a 100 trials of arbitrary $2 \times 2 \times 2$ tensors.

\subsection{Perturbations of singular values}
When an $m\times n$ matrix $A$ is additively perturbed by a matrix $E$ to form $\hat{A}=A+E$, then Weyl's Theorem \cite{Stewart90perturbationtheory} bounds the absolute perturbations of the corresponding singular values by
$$
|\sigma_{i} - \hat{\sigma}_i | \leq ||E||_2,
$$
where the $\hat{\sigma}_i$'s are the singular values of $\hat{A}$. It is possible to extend Weyl's Theorem to the TTr1 decomposition of the perturbed tensor $\hat{\ten{A}}=\ten{A}+\ten{E}$. Suppose we want to determine an upper bound for the perturbation of one of the singular values $\tilde{\sigma}$. We first introduce the simpler notation
$$
\tilde{\sigma}=\sigma_1\,\sigma_2\,\cdots\,\sigma_{d-2},
$$
where $\sigma_k (k=1,\ldots,d-2)$ denotes the singular value at level $k$ in the branch of the TTr1-tree corresponding with $\tilde{\sigma}$. Applying Weyl's Theorem to the first factor gives
$$
|\sigma_1-\hat{\sigma}_1| \leq ||\bar{\ten{E}}||_2,
$$
which we can rewrite into
\begin{equation}
|\hat{\sigma}_1| \leq |\sigma_1|+||\bar{\ten{E}}||_2.
\label{eq:firstsigma}
\end{equation}
Each of the remaining factors $\sigma_2,\ldots,\sigma_{d-2}$ are the singular values of a reshaped right singular vector $v_1,\ldots,v_{d-3}$. Again, application of Weyl's Theorem allows us to write
\begin{align}
\nonumber |\hat{\sigma}_k| & \leq |\sigma_k|+||\Delta\bar{v}_{k-1}||_2 \quad (k=2,\ldots,d-2),\\
  & \leq |\sigma_k|+||\Delta\bar{v}_{k-1}||_F = |\sigma_k|+||\Delta{v}_{k-1}||_2 \quad (k=2,\ldots,d-2).
\label{eq:remainingsigma}
\end{align}
An upper bound for the $||\Delta\bar{v}_{k-1}||_2$ term is difficult to derive. Fortunately, it is possible to replace the $||\Delta\bar{v}_{k-1}||_2$ term by $||\Delta\bar{v}_{k-1}||_F=||\Delta{v}_{k-1}||_2$, for which first-order approximations exist \cite{Liu2008}. Multiplying \eqref{eq:firstsigma} with \eqref{eq:remainingsigma} over all $k$ we obtain
$$
|\hat{\sigma}_1| \cdots |\hat{\sigma}_{d-2}| \leq (|\sigma_1|+||\bar{\ten{E}}||_2) \cdots (|\sigma_k|+||\Delta{v}_{k-1}||_2),
$$
which can be simplified by ignoring higher order terms to
\begin{equation}
|\hat{\sigma}_1| \cdots |\hat{\sigma}_{d-2}| \leq |\sigma_1| \cdots |\sigma_{d-2}| +|\sigma_2| \cdots |\sigma_{d-2}|\,||\bar{\ten{E}}||_2 + \sum_{k=2}^{d-2} (\prod_{i\neq k} |\sigma_i|) ||\Delta{v}_{k-1}||_2.
\label{eq:perturbsigma}
\end{equation}
The maximal value for $\sigma_2, \ldots, \sigma_{d-2}$ is 1 and hence \eqref{eq:perturbsigma} can be written as
$$
|\hat{\sigma}_1| \cdots |\hat{\sigma}_{d-2}| \leq |\sigma_1| \cdots |\sigma_{d-2}| + ||\bar{\ten{E}}||_2 + \sigma_1 (\sum_{k=1}^{d-3} ||\Delta{v}_{k}||_2).
$$
Hence we arrive at the expression
\begin{equation}
|\tilde{\sigma}-\hat{\tilde{\sigma}}| \leq  ||\bar{\ten{E}}||_2 + \sigma_1 (\sum_{k=1}^{d-3} ||\Delta{v}_{k}||_2),
\label{eq:weylTTr1}
\end{equation}
which generalizes Weyl's Theorem to the TTr1 decomposition by the addition of a correction term $\sigma_1 (\sum_{k=1}^{d-3} ||\Delta{v}_{k}||_2)$. This correction term depends on the largest singular value of the first level and the perturbations on the right singular vectors for levels 1 up to $d-3$.

\subsection{Conversion to the Tucker decomposition}
It is possible to convert the sum of orthogonal rank-1 terms obtained from a (truncated) TTr1 decomposition into the Tucker decomposition
$$
\ten{A} \;=\;\ten{S} {_{\times_1}} U_1  {_{\times_2}} U_2 \cdots {_{\times_d}} U_d
$$
where $\ten{S} \in \mathbb{R}^{r_1 \times r_2 \times \cdots \times r_d}$ is called a core tensor and $U_k \in \mathbb{R}^{n_k \times r_k}$ are orthogonal factor matrices. In this way it becomes relatively easy to compute an approximation of a tensor with a known approximation error in the Tucker format with orthogonal factor matrices. This conversion is easily achieved using simple matrix operations. To avoid notational clumsiness, we illustrate the conversion from the TTr1 representation into the Tucker form through the specific TTr1 decomposition in~\eqref{ex:Xttr1}. Suppose only alternate terms in the equation are significant and we therefore only keep the $3$ terms associated with $\sigma_1 \sigma_{11}$, $\sigma_2 \sigma_{21}$ and $\sigma_3\sigma_{31}$. Then, the mode vectors of these outer products are collected and subjected to economic QR factorization. We note that $U_1=[u_1~u_2~u_3]\in\mathbb{R}^{3\times 3}$ is already orthogonal and does not need to go through a QR factorization, while%
\begin{align*}
  \left[\!\! {\begin{array}{*{20}c}
   {u_{11} } & {u_{21} } & {u_{31} }  \\
 \end{array} } \!\!\right] &= \underbrace{\left[\!\! {\begin{array}{*{20}c}
   {q_{21} } & {q_{22} } & {q_{23} }  \\
 \end{array} } \!\!\right]}_{U_2}\left[\!\! {\begin{array}{*{20}c}
   {1} & {\alpha _{12} } & {\alpha _{13} }  \\
   0 & {\alpha _{22} } & {\alpha _{23} }  \\
   0 & 0 & {\alpha _{33} }  \\
 \end{array} } \!\!\right] \nonumber,\\
  \left[\!\! {\begin{array}{*{20}c}
   {v_{11} } & {v_{21} } & {v_{31} }  \\
 \end{array} } \!\!\right] &= \underbrace{\left[\!\! {\begin{array}{*{20}c}
   {q_{31} } & {q_{32} }  \\
 \end{array} } \!\!\right]}_{U_3}\left[\!\! {\begin{array}{*{20}c}
   {1} & {\beta _{12} } & {\beta _{13} }  \\
   0 & {\beta _{22} } & {\beta _{23} }  \\
 \end{array} } \!\!\right].%
\end{align*}
Consequently, the truncated TTr1SVD of~\eqref{ex:Xttr1} reads%
\begin{align}
\ten{A}\approx &\sigma_1 \sigma_{11} (u_1\circ u_{11}\circ v_{11}) + \sigma_2 \sigma_{21} (u_2\circ u_{21}\circ v_{21})+ \sigma_3 \sigma_{31} (u_3\circ u_{31}\circ v_{31})\nonumber\\
=& \sigma_1 \sigma_{11} (u_1\circ q_{21}\circ q_{31})+\sigma_2 \sigma_{21} (u_2\circ (\alpha_{12} q_{21}+\alpha_{22}q_{22})\circ (\beta_{12}q_{31}+\beta_{22}q_{32}))+\nonumber\\
& \sigma_3 \sigma_{31} (u_3\circ (\alpha_{13} q_{21}+\alpha_{23} q_{22}+\alpha_{33} q_{23}) \circ (\beta_{13} q_{31}+\beta_{23} q_{32}))\nonumber\\
=& \ten{S} {_{\times_1}} U_1  {_{\times_2}} U_2 {_{\times_3}} U_3,
\label{eqn:ttr1svd2tucker}
\end{align}
where the core tensor $\ten{S}\in\mathbb{R}^{3\times 3\times 2}$ is filled with coefficients found through expanding the outer products and collecting terms in~(\ref{eqn:ttr1svd2tucker}). Observe that the dimensions of the core tensor $\ten{S}$ are completely determined by the ranks of the orthogonal factor matrices. From practical examples we observe that the Tucker core obtained in this way is more sparse compared to the Tucker core computed from the ALS algorithm~\cite{TTB_Sparse,TTB_Software}. If we take for example a random tensor in $\in \mathbb{R}^{4 \times 3 \times 15}$, then its TTr1 decomposition consists of 12 terms. The ranks of the orthogonal factor matrices $U_1,U_2,U_3$ are then $4,3,12$, respectively. Consequently, we have that $\ten{S} \in \mathbb{R}^{4 \times 3 \times 12}$ with $3+3\cdot9=30$ nonzero entries. In contrast, computing the Tucker decomposition using the ALS method results in a maximally dense core tensor of 144 nonzero entries.



\section{Numerical examples}
\label{sec:examples}

In this section, we demonstrate some of the properties of the TTr1 decomposition and compare it with the ALS-CP and Tucker decompositions by means of numerical examples. All experiments are done in MATLAB on a desktop computer. A Matlab/Octave implementation of the TTr1SVD algorithm can be freely downloaded and modified from \url{https://github.com/kbatseli/TTr1SVD}. The ALS-CP and Tucker decompositions are computed by the ALS optimization tool provided in the MATLAB Tensor Toolbox~\cite{TTB_Sparse,TTB_Software}. All ALS procedures are fed by random initial guesses, therefore their errors are defined as the average error over multiple executions. 

\subsection{Example 1: Singular values and permutation of indices}
\label{ex:ex1}
We start with tensor $\ten{A}$ in Figure~\ref{fig:A}. Since it is discussed in Section~\ref{subsec:upperbound} that TTr1 decomposition depends on the ordering of the indices, we demonstrate the TTr1 decomposition with different permutations of the indices. For a 3-way tensor, the order of indices can be $\{i_1,i_2,i_3\}$, $\{i_1,i_3,i_2\}$, $\{i_2,i_1,i_3\}$, $\{i_2,i_3,i_1\}$, $\{i_3,i_1,i_2\}$ or $\{i_3,i_2,i_1\}$. Since the order of the last 2 indices will not affect the $\sigma$'s in TTr1 decomposition, we only list the $\sigma$'s under the permutations $\{i_1,i_2,i_3\}$, $\{i_2,i_3,i_1\}$ and $\{i_3,i_1,i_2\}$ in Table~\ref{tab:ttr1_X}, in descending order. As a result, although permutations of indices may give different upper bounds on the rank, TTr1 decomposition still outputs the same rank$(\ten{A})=4$ in all permutations. The largest (dominant) singular value (69.6306) differs only slightly with respect to the permutations, which is also the general observation. Note that the singular values are quite similar over all permutations for this example. It can also be seen that some singular values are numerically zero. The same threshold commonly used to determine the numerical rank of a matrix using the SVD can also be used for the TTr1 decomposition.
\begin{table}[ht]
\centering
\caption{$\tilde{\sigma}$'s of TTr1 decomposition for $\ten{A}$.}
\label{tab:ttr1_X} 
\begin{tabular}{@{}lrrr@{}}
Order of indices & $\{i_1,i_2,i_3\}$ &  $\{i_2,i_3,i_1\}$ & $\{i_3,i_1,i_2\}$ \\ \midrule
$\tilde{\sigma}_1$ & 69.6306 & 69.6306 & 69.6306 \\
$\tilde{\sigma}_2$ & 6.9190 & 6.9551 & 6.9567 \\
$\tilde{\sigma}_3$ & 1.8036 & 1.6108 & 1.6010 \\
$\tilde{\sigma}_4$ & 0.6729& 0.7781 & 0.7840\\
$\tilde{\sigma}_5$ & 6.7e-15 & 1.9e-15 & 4.3e-15\\
$\tilde{\sigma}_6$ & 1.3e-15 & 1.9e-15 & 1.4e-15\\
$\tilde{\sigma}_7$ & NA& 7.3e-16 & NA\\
$\tilde{\sigma}_8$ & NA& 5.3e-16 & NA\\
\end{tabular}
\end{table}
An interesting consequence of the rank-deficiency of $\ten{A}$ is that we can interpret the rank-1 terms corresponding with the very small numerical singular values as being related to $\spn(\ten{A})^\perp$. For the $\{i_1,i_2,i_3\}, \{i_3,i_1,i_2\}$ permutations there are 2 extra orthogonal complement tensors while for the $\{i_2,i_3,i_1\}$ there are 4 extra orthogonal complement tensors. Figure \ref{fig:SVperms} shows the similar singular value curves for all 6 permutations of a random $3\times 4 \times 2$ tensor, where it can also be seen that there are basically 3 distinct permutations and that the largest singular values only differ slightly over all permutations.

\begin{figure}[th]
\begin{center}
\includegraphics[width=.85\textwidth]{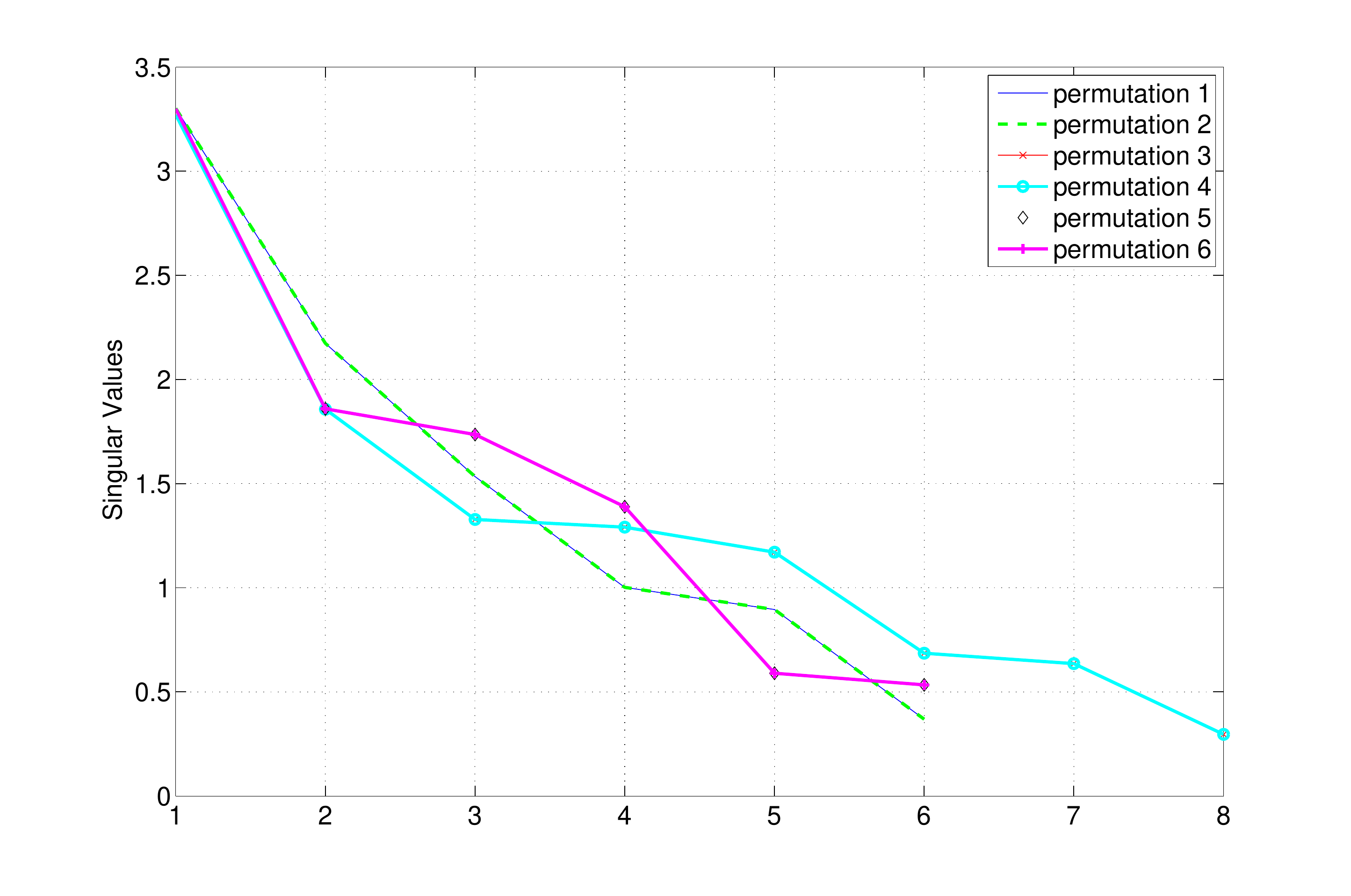}
\caption{Singular value curves for each of the 6 permutations of a random $3\times 4 \times 2$ tensor.}
\label{fig:SVperms}
\end{center}
\end{figure}

\subsection{Example 2: Comparison with ALS-CP and Tucker decomposition}
Next, ALS-CP is applied on $\ten{A}$. To begin with, we compute the best rank-1 approximation of $\ten{A}$. ALS-CP gives the same weight 69.6306 as the TTr1 decomposition, implying that both decompositions result in the same approximation in terms of the Frobenius norm. The errors between $\ten{A}$ and its approximations $\tilde{\ten{A}}$, computed using the ALS-CP and TTr1SVD method, are listed in Table~\ref{tab:CP_X} for increasing rank.
\begin{table}[ht]
\centering
\caption{Errors $||\ten{A}-\tilde{\ten{A}}||_F$ of ALS-CP and TTr1SVD for increasing rank $R$.}
\label{tab:CP_X} 
\begin{tabular}{@{}lrrrrr@{}}
Rank & 1 &  2 & 3 & 4 & 5\\ \midrule
TTr1SVD & 7.2 & 1.9 & 0.7 & 6.8e-15 & 1.3e-15\\
ALS-CP & 7.2 & 0.8 & 3.6e-2 & 1.4e-10 & 4.7e-11\\
\end{tabular}
\end{table}

Table~\ref{tab:CP_X} confirms that the $\textrm{rank}_\perp(\ten{A}) = \textrm{rank}(\ten{A}) = 4$. It also indicates that as an optimization approach, ALS-CP itself cannot determine the rank, but only the best rank-$R$ approximation for a specific $R$. Furthermore, it should be noticed that the TTr1 decomposition can always give an $R$-rank-1-term approximation with orthonormal outer products, while ALS-CP cannot assure this property. Finally, a Tucker decomposition with a core size $(2,2,2)$ is applied on $\ten{A}$. The resulting dense core tensor $\ten{S}\in \mathbb{R}^{2\times 2 \times 2}$ is given by
$$
{\ten{S}}_{i_1 i_2 1} =
\begin{pmatrix}
69.6306 & -0.0181\\
-0.0701 & -0.7840
\end{pmatrix} \quad
{\ten{S}}_{i_1 i_2 2} =
\begin{pmatrix}
-0.0113 & -6.9190\\
-1.6108 & -0.7010
\end{pmatrix}.
$$ 
The rank-1 outer factors obtained from the Tucker decomposition are also orthonormal. However, compared to the TTr1 decomposition, the Tucker format needs twice the number of factors than TTr1.

\subsection{Example 3: Rank behavior under largest rank-1 term subtraction}
In this example we investigate the behavior of the singular value curves and the rank when the largest rank-1 term obtained from the TTr1 decomposition is consecutively subtracted. This means that we start with $\ten{A}$ from Figure \ref{fig:A}, compute its largest orthogonal rank-1 term $\ten{T}_1$ from the TTr1 decomposition and subtract it to obtain $\ten{A}-\ten{T}_1$, after which the procedure is repeated. Figure \ref{fig:greedy} shows the singular value curves from the TTr1 decompositions obtained for each of the iterations, where it is easily seen that each curve gets shifted to the left with each iteration. In other words, the largest singular value in the next iteration is the second largest singular value of the previous iteration, etc.. It is also clear that subtracting the largest orthogonal rank-1 term does not necessarily decrease the rank as also described in \cite{Stegeman2010}. Indeed, using a numerical threshold of $\textrm{max}(4,8)\cdot 1.11 \times 10^{-16}\cdot \sigma_1=1.13\times 10^{-13}$ on the singular values obtained in the first iteration will still return a numerical orthogonal rank of 4.

In addition, the rank of the obtained tensors in each iteration was also determined from the CP-decomposition. From Table~\ref{tab:CP_X}, a numerical threshold of $10^{-10}$ was set to the absolute error $||\ten{A}-\tilde{\ten{A}}||_F$ to determine the CP-rank. In Table \ref{tab:CP_rank}, both the CP-rank and the orthogonal rank from the TTr1 decomposition are compared. It can be seen that the rank determined from ALS-CP increases while the orthogonal rank monotonically decreases. In this sense, the orthogonal rank appears to be more robust under largest rank-1 term subtraction.

\begin{figure}[th]
\begin{center}
\includegraphics[width=.85\textwidth]{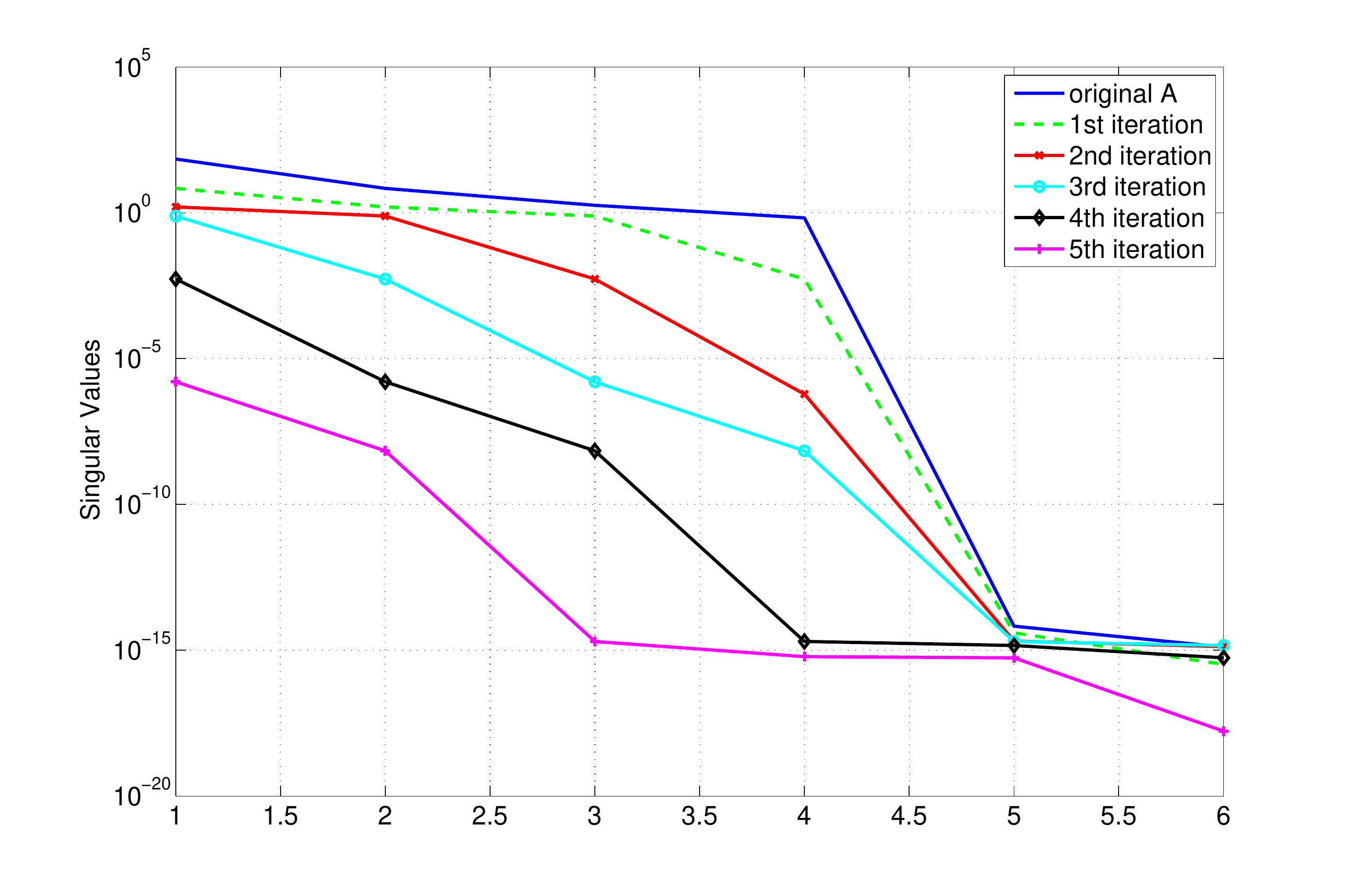}
\caption{Singular value curves for consecutive subtractions of the largest rank-1 term.}
\label{fig:greedy}
\end{center}
\end{figure}

\begin{table}[ht]
\centering
\caption{Ranks determined by ALS-CP and TTr1 decomposition for consecutive subtraction of largest rank-1 term.}
\label{tab:CP_rank} 
\begin{tabular}{@{}lrrrrrr@{}}
Iteration & 0 &  1 & 2 & 3 & 4 & 5\\
\hline
CP-rank & 4 & 4 & 4 & 5 & 3 & 2\\
TTr1 rank & 4 & 4 & 4 & 4 & 3 & 2
\end{tabular}
\end{table}

\subsection{Example 4: Perturbation of the singular values}
In this example we illustrate the robustness of the computed singular values of our running example tensor $\ten{A}$ when it is subjected to additive perturbations. We construct a perturbation tensor $\ten{E} \in \mathbb{R}^{3 \times 4 \times 2}$ where each entry is drawn from a zero mean Gaussian distribution with variance $10^{-6}$. We then compute the following two norms of $\ten{E}$ and $\bar{\ten{E}}$
$$
||\ten{E} ||_F \;=\; 5.48\times 10^{-6} \quad \textrm{ and } \quad ||\bar{\ten{E}} ||_2 \;=\; 4.13\times 10^{-6},
$$
where $\bar{\ten{E}}$ is $\ten{E}$ reshaped into a $3 \times 8$ matrix. Comparing the perturbed singular values $\bar{\tilde{\sigma}}_1,\ldots,\bar{\tilde{\sigma}}_6$ of $\ten{A}+\ten{E}$ with the singular values $\tilde{\sigma}_1,\ldots,\tilde{\sigma}_6$ then shows that
\begin{equation}
\sqrt{ (\bar{\tilde{\sigma}}_1-\tilde{\sigma}_1)^2 + \cdots + (\bar{\tilde{\sigma}}_6-\tilde{\sigma}_6)^2 } = 3.78\times 10^{-6} < || \ten{E} ||_F,
\label{ex:Mirsky}
\end{equation}
and
\begin{equation}
|\bar{\tilde{\sigma}}_i-\tilde{\sigma}_i | \leq ||\bar{\ten{E}} ||_2 \quad (i=1,\ldots,6).
\label{ex:Weyl}
\end{equation}
These two inequalities \eqref{ex:Mirsky} and \eqref{ex:Weyl} are very reminiscent of Mirsky's and Weyl's Theorem \cite{Stewart90perturbationtheory}, respectively, for the perturbation of singular values for matrices.

\subsection{Example 5: Gradual decrease of intermediate singular value products}
In the discussion on reducing the total number of required SVDs it was shown that the product of the singular values along a branch becomes smaller and smaller for every additional level. In this example we demonstrate this gradual decrease for a random $2\times 2 \times 2 \times 2 \times 2$ tensor where each entry is drawn from a zero mean Gaussian distribution with variance 1. The TTr1 decomposition always has 16 rank-1 terms. Figure \ref{fig:sigdecrease} shows the intermediate singular value products $\sigma_i \sigma_{ij} \cdots \sigma_{ij\cdots m}$ as a function of the level for $\tilde{\sigma}_1,\tilde{\sigma}_8,\tilde{\sigma}_{13},\tilde{\sigma}_{16}$. On the figure it can be seen that the intermediate singular value products indeed decrease as the level increases. The TTr1-tree for this tensor is a binary tree. Each SVD of a $v$ vector therefore produces 2 singular values. It is consistently observed that of the two singular values of $\bar{v}$, one is very close to unity, with values around $0.8$ or $0.9$. The other singular value typically has values around $0.5$. Branches of the tree that mostly choose the singular value close to unity therefore exhibit a very slow decrease while branches that predominantly choose the smaller singular value decrease faster. This is seen in Figure \ref{fig:sigdecrease} as a bigger descent of the intermediate products of $\tilde{\sigma_8},\tilde{\sigma}_{16}$ compared to $\tilde{\sigma}_1,\tilde{\sigma}_{13}$.

\begin{figure}[ht]
\centering
\includegraphics[width=.9\textwidth]{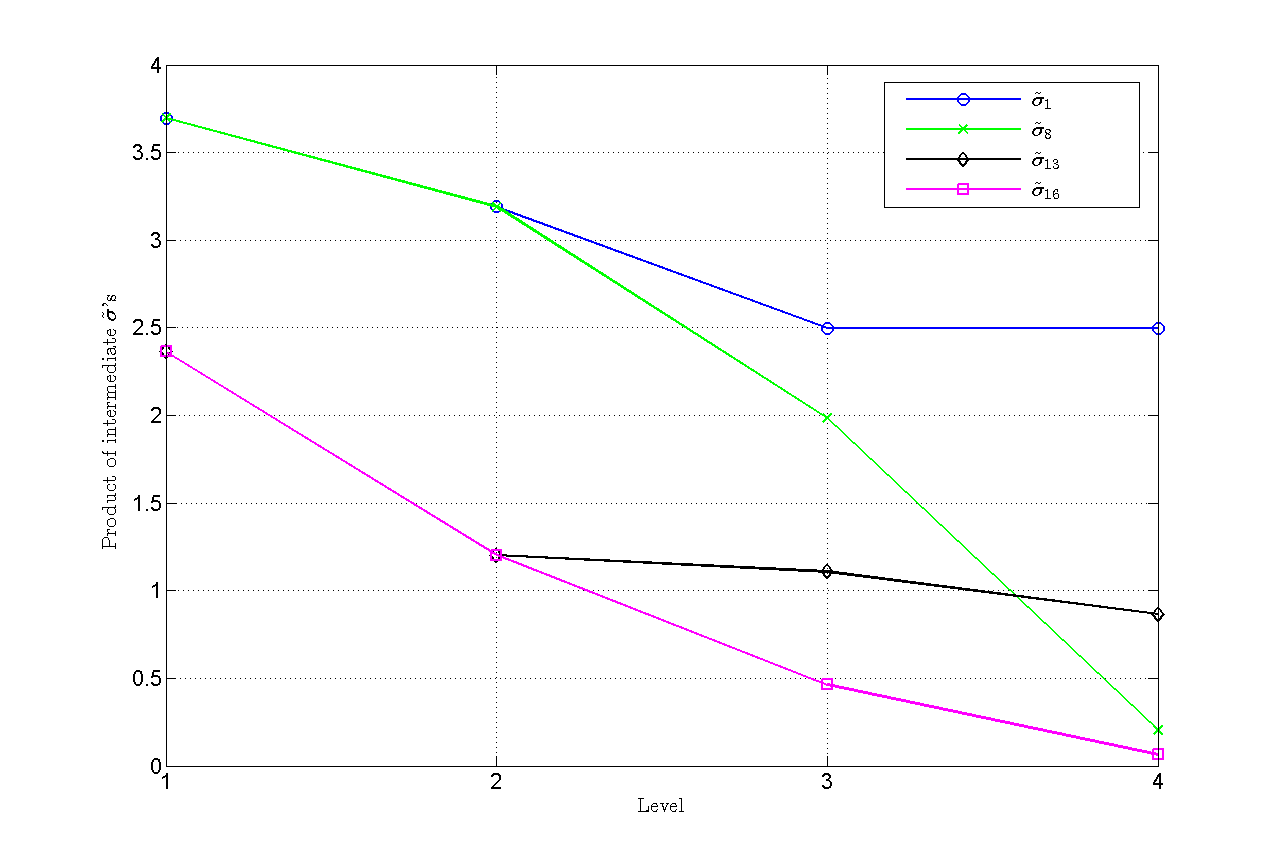}
\caption{Gradual decrease of intermediate singular value products as a function of the level.}
\label{fig:sigdecrease}
\end{figure}

\subsection{Example 6: Exponential decaying singular values}
In this example we illustrate the computation of an approximation $\tilde{\ten{A}}$ using Lemma \ref{lem:errorlem} when the singular values $\tilde{\sigma{}}$ decay exponentially. Consider the tensor $\ten{A} \in \mathbb{R}^{5 \times 5 \times 5}$ with
$$
\ten{A}_{i_1i_2i_3} \;=\; \frac{1}{i_1+i_2+i_3},
$$
which has very smoothly decaying singular values as shown in Figure \ref{fig:ex6}. There are a total number of 25 rank-1 terms in the TTr1 decomposition. Suppose we are interested in obtaining an approximation $\tilde{\ten{A}}$ such that $\||\ten{A}-\tilde{\ten{A}}||_F \leq 10^{-6}$. Sorting the rank-1 terms by descending singular values and using Lemma \ref{lem:errorlem}, the approximation would then consist of 17 terms since
$$
 \tilde{\sigma}_{16} =  \num{3.18}\times 10^{-6} \geq \tilde{\sigma}_{17} =  \num{1.18}\times 10^{-6} \geq \num{1.00}\times 10^{-6} \geq \tilde{\sigma}_{18} = \num{9.00}\times 10^{-7}.
$$
The approxi-rank gap $\tilde{\sigma}_{17}/\tilde{\sigma}_{18}=1.31$, which indicates that there is no clear ``gap" between $\tilde{\sigma}_{17}$ and $\tilde{\sigma}_{18}$. In contrast, the tensor in Example 1 has an approxi-rank gap of $\tilde{\sigma}_{4}/\tilde{\sigma}_{5} \approx 10^{14}$. Also note that it is not possible to reduce the number of SVDs during execution of the TTr1SVD algorithm since none of the first five computed singular values $\sigma_1,\ldots,\sigma_5$ satisfy condition \eqref{eq:reducecond}, with $\sigma_{k0}=1$. Next, the approximations obtained from the TTr1 decomposition and CANDECOMP of this tensor are compared for increasing rank. The CANDECOMP was computed over 10 trials with different initial guesses using the CP-ALS method. The absolute errors in terms of the rank are listed in Table \ref{tab:ecomp}. For the CANDECOMP case the mean absolute error over the 10 trials is reported. From Table \ref{tab:ecomp} it is seen that the errors are almost identical up to the first 24 terms. Since the TTr1 decomposition consists of 25 terms, the error drops at that term to the order of the machine precision, while the ALS-CP method fails to produce any significant improvement in the error. Even when a CANDECOMP of 100 rank-1 terms are computed, the average absolute error is around $10^{-12}$. 

\begin{figure}[ht]
\centering
\includegraphics[width=.9\textwidth]{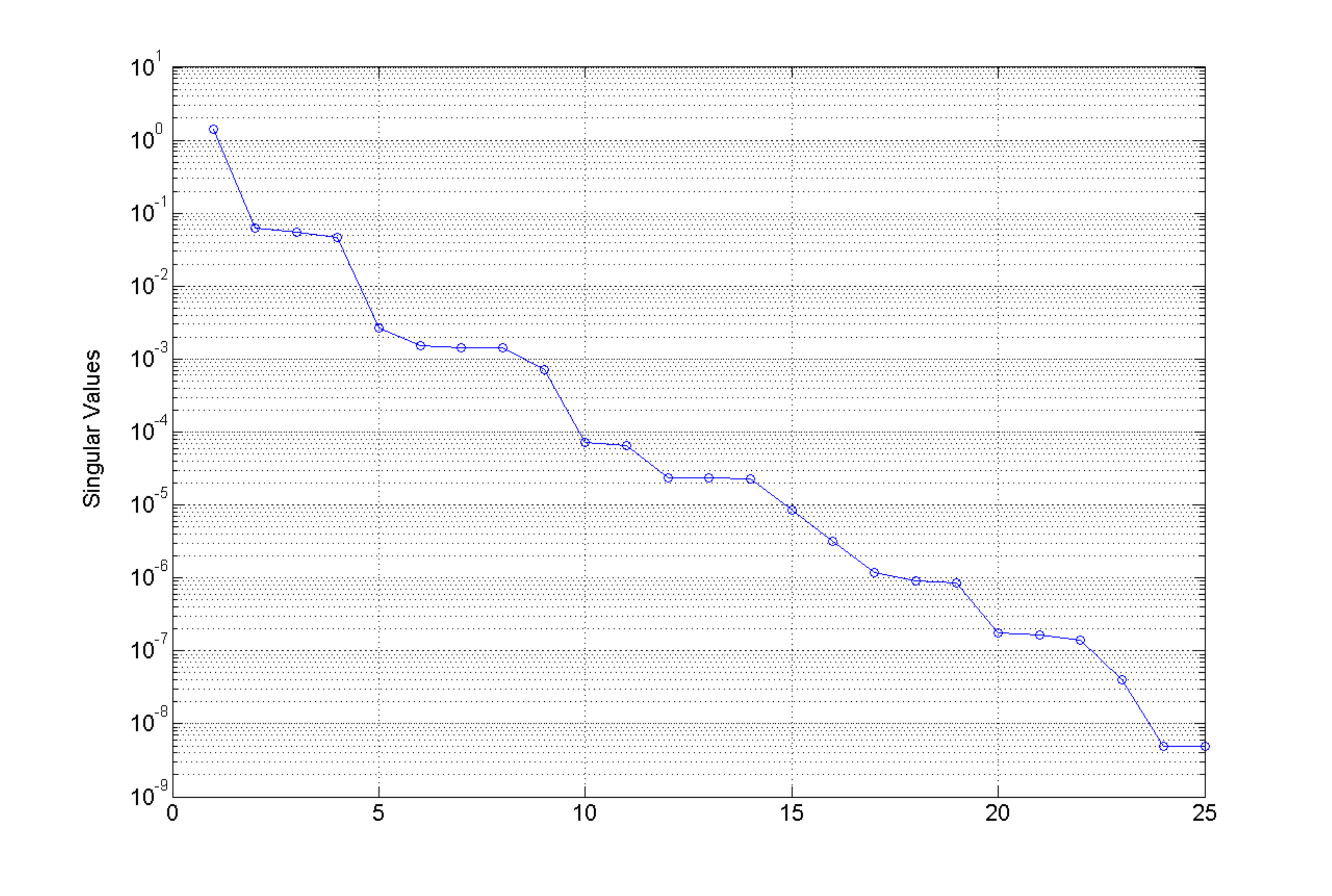}
\caption{Singular value decay of the function-generated tensor $\ten{A}_{i_1i_2i_3} \;=\; 1/(i_1+i_2+i_3)$.}
\label{fig:ex6}
\end{figure}

\begin{table}[ht]
\centering
\caption{Errors $||\ten{A}-\tilde{\ten{A}}||_F$ of ALS-CP and TTr1SVD for increasing rank $R$.}
\label{tab:ecomp} 
\begin{tabular}{@{}lrrrrrrr@{}}
Rank & 1 &  5 & 10 & 15 & 20 & 25 \\\midrule
TTr1SVD & \num{9.5e-2}&\num{2.6e-3}&\num{7.6e-5}&\num{3.6e-6}&\num{2.2e-7}&\num{9.9e-16}\\
ALS-CP & \num{9.5e-2}&\num{1.6e-3}&\num{3.6e-5}&\num{4.5e-6}&\num{1.2e-7}&\num{1.6e-7}
\end{tabular}
\end{table}

\section{Conclusion}
\label{sec:conclusions}
In this paper, a constructive TTr1 decomposition algorithm, named TTr1SVD, was proposed to decompose high-order real tensors into a finite sum of real orthogonal rank-1 outer products. Compared to existing CP approaches, the TTr1 decomposition has many favorable properties such as uniqueness, easy quantification of the approximation error, and an easy conversion to the Tucker format with a sparse core tensor. A complete characterization of all tensors orthogonal to the original tensor was also provided for the first time, which is readily available via the TTr1SVD and easily visualized by an outer product column table. As an application example, this outer product column table was used to provide an elegant constructive proof of the maximum rank-3 property of $2\times 2\times 2$ tensors over the real field. Numerical examples verified and demonstrated the favorable properties of TTr1SVD in decomposing and analyzing real tensors. 

\section*{Acknowledgements}
The authors would like to thank the Associate Editor Pierre Comon and the anonymous referees for the many constructive comments.

 \bibliographystyle{siam}
 \bibliography{references.bib}

\end{document}